\documentclass[12pt]{amsart}

\usepackage{url}
\usepackage{graphicx}
\usepackage{amsmath}
\usepackage{amscd}
\usepackage{amsfonts}
\usepackage{amssymb}
\usepackage{color}
\usepackage{fullpage}
\usepackage{mathtools}
\usepackage{todonotes}
\usepackage{enumitem}
\usepackage{comment}
\usepackage{xparse}
\usepackage{nicefrac}
\usepackage[pagebackref,hypertexnames=false, colorlinks, citecolor=red,linkcolor=blue, urlcolor=red]{hyperref}
\usepackage{cleveref}

\numberwithin{equation}{section}

\newtheorem{theorem}{Theorem}[section]
\newtheorem{lemma}[theorem]{Lemma}
\newtheorem{proposition}[theorem]{Proposition}

\newtheorem{quest}[theorem]{Question}
\newtheorem{claim}[theorem]{Claim}
\newtheorem{corollary}[theorem]{Corollary}

\theoremstyle{definition}
\newtheorem{example}[theorem]{Example}
\newtheorem{remark}[theorem]{Remark}
\newtheorem{definition}[theorem]{Definition}
\newtheorem{problem}[theorem]{Problem}

\newcommand{\be}{\begin{equation}}
\newcommand{\ee}{\end{equation}}
\newcommand{\bes}{\begin{equation*}}
\newcommand{\ees}{\end{equation*}}

\newcommand{\ignore}[1]{}

\renewcommand{\Re}{\operatorname{Re}}

\newcommand{\Aut}{\operatorname{Aut}}

\DeclareDocumentCommand{\norm}{g} 
{\IfNoValueF{#1}{\lVert #1\rVert}
	\IfNoValueT{#1}{\lVert \cdot\rVert}}
\newcommand\av[1]{\left|#1\right|}
\DeclareDocumentCommand{\inp}{g} 
{\IfNoValueF{#1}{\langle \, #1 \, \rangle}
	\IfNoValueT{#1}{\langle \, \cdot\, ,\, \cdot \, \rangle}} 
\newcommand{\pare}[1]{\left({#1}\right)} 

\newcommand{\bB}{\mathbb{B}}

\newcommand{\bH}{\mathbb{H}}

\newcommand{\R}{\mathbb{R}}	    
\newcommand{\C}{\mathbb{C}}	    
\newcommand{\N}{\mathbb{N}}	    
\newcommand{\D}{\mathbb{D}}     

\newcommand{\cB}{\mathcal{B}}   
\newcommand{\cD}{\mathcal{D}}

\newcommand{\cH}{\mathcal{H}}

\newcommand{\cK}{\mathcal{K}}

\newcommand{\cS}{\mathcal{S}}

\newcommand\set[1]{\left\{#1\right\}}

\newcommand\on[1]{\operatorname{#1}}

\newcommand{\spa}{\on{span}}
\newcommand{\sm}{\smallsetminus}

\newcommand{\onto}{\xymatrix{\ar@{>>}[r]&}}
\newcommand{\ol}{\overline}

\DeclareDocumentCommand{\eqn}{d[]m}{
\begin{equation}
	\IfNoValueF{#1}{\label{eq:#1}}
	#2
\end{equation}
}

\newcommand{\figureref}[1]{\hyperref[#1]{Figure~\ref*{#1}}}
\newcommand{\tableref}[1]{\hyperref[#1]{Table~\ref*{#1}}}
\newcommand{\chapterref}[1]{\hyperref[#1]{chapter~\ref*{#1}}}
\newcommand{\Chapterref}[1]{\hyperref[#1]{Chapter~\ref*{#1}}}
\newcommand{\Chdotref}[1]{\hyperref[#1]{Ch.~\ref*{#1}}}
\newcommand{\appendixref}[1]{\hyperref[#1]{appendix~\ref*{#1}}}
\newcommand{\Appendixref}[1]{\hyperref[#1]{Appendix~\ref*{#1}}}
\newcommand{\sectionref}[1]{\hyperref[#1]{section~\ref*{#1}}}
\newcommand{\subsectionref}[1]{\hyperref[#1]{\S~\ref*{#1}}}
\newcommand{\exerciseref}[1]{\hyperref[#1]{Exercise~\ref*{#1}}}
\newcommand{\exampleref}[1]{\hyperref[#1]{Example~\ref*{#1}}}
\newcommand{\thmref}[1]{\hyperref[#1]{Theorem~\ref*{#1}}}
\newcommand{\propref}[1]{\hyperref[#1]{Proposition~\ref*{#1}}}
\newcommand{\lemmaref}[1]{\hyperref[#1]{Lemma~\ref*{#1}}}
\newcommand{\corref}[1]{\hyperref[#1]{Corollary~\ref*{#1}}}
\newcommand{\defnref}[1]{\hyperref[#1]{Definition~\ref*{#1}}}
\newcommand{\remarkref}[1]{\hyperref[#1]{Remark~\ref*{#1}}}

{\par\noindent\rule{\textwidth}{0.5pt}
\end{figure}}

\newcommand{\Mult}{\operatorname{Mult}}

\begin{document}
\title{Weighted composition operators and classification of weighted Hardy spaces}
\author{David Masuda}
\address{D.W., Technion Israel Institute of Technology\\
 Technion City, Haifa\; 3200003\\
 Israel}
\email{masudadavid@campus.technion.ac.il}


\addcontentsline{toc}{section}{Abstract}
\begin{abstract}
In this paper, we shed light on basic questions such as:  
What is the ``right'' underlying set on which a function space lives, and is this set uniquely determined? 
What are the appropriate morphisms to classify function spaces? When does a weighted composition operator induce an isomorphism, and when does it induce an isometric isomorphism?
We first formulate suitable notions for domains of reflexive functional Banach spaces and study the corresponding notions of isomorphism between Hilbert function spaces. We then apply this framework to classify certain unitarily invariant spaces of holomorphic functions in several variables.
Our main results extend Hartz's results from the complete Pick setting to the general setting of weighted Hardy spaces and establish a simple relation between the kernels of isomorphic spaces, improving also on earlier work of Ofek and Sofer in several directions.
Additionally, we apply our results to investigate and estimate the Banach–Mazur distance of Hilbert function spaces.
\end{abstract}

\maketitle

\section{Introduction}
Hilbert spaces of functions form a natural meeting point of complex analysis and operator theory. 
It is a long-standing objective to understand their structure and to classify natural families of such spaces.
In this paper, we focus on the weighted Hardy spaces, which are unitary-invariant Hilbert function spaces on a ball in $\C^d$ that contain polynomials as a dense subset. This class of spaces includes many classical examples, such as the Drury--Arveson spaces \cite{hartz2023invitation, hartz2013operator} and the one-parameter family $\cH_s$, which we introduce briefly at the end of Section 5 and contains the Bergman space, the Hardy space, and the Dirichlet space. 

A meaningful classification of Hilbert function spaces must retain some of their function-theoretic structure. This leads naturally to the problem of identifying the appropriate morphisms between such spaces.
Henceforth, traditionally, the relationship between Hilbert function spaces is often studied via an RKHS isomorphism, whose adjoint is a weighted composition operator \cite{hartz2017isomorphism, OFEK, Gilad, salomon2016isomorphism, watted2025deformations}. 
Weighted composition operators have been studied much more extensively, primarily as operators acting on a single space, owing to their importance in the study of Banach spaces of functions \cite{CowenCarlC, hartz2025weighted, kumari2025composition, le2012self, lefevre2025characterization,  martin2019co, MR4225497}.

A well-known result (Proposition \ref{prop : isoalg}) reveals that an RKHS isomorphism induces an isomorphism between the multiplier algebras. 
On the one hand, in \cite[Corollary 9.4]{davidson2015multipliers}, Davidson, Hartz, and Shalit showed that the multiplier algebras $\Mult(\cH_s)$ for $s\le 0$ are mutually non-isomorphic, and hence the spaces $\cH_s$ are mutually non-isomorphic via RKHS isomorphism.
On the other hand, when looking outside the class of complete Pick spaces, one finds that two multiplier algebras can be isometrically isomorphic, yet the Hilbert function spaces that give rise to them need not be isomorphic (even in the weakest sense of this paper).
A good example of this phenomenon is the relation between the Hardy space and the Bergman space. Both have $H^\infty\pare{\D}$ as their multiplier algebra, and, in fact, the identity map from $H^\infty\pare{\D}$ to itself is a (completely isometric) isomorphism implemented by composition with the identity map on the disc. 
However, we show (in two ways) that the Hardy and Bergman spaces are not isomorphic through an invertible weighted composition operator.

While the inverse of an RKHS isomorphism is always an RKHS isomorphism, this is not the case for invertible weighted composition operators in general.
We illustrate this in several examples, where the most interesting is taken from \cite{McCarthy_2017}; see Example \ref{ex:Orr}. 
To keep the problem interesting, we should consider the basic example of the Hardy space restricted to $X=\set{\nicefrac{1}{k}}_{k=2}^\infty$.
In this situation, the map $C_\iota:f\mapsto f|_X$ is a bounded invertible composition operator whose inverse is not a weighted composition operator.
We identify sufficient conditions under which a weighted composition operator is the adjoint of an RKHS isomorphism. 
To this end, we define a maximal domain for reflexive functional Banach spaces by combining related definitions from \cite{MR3687947,CowenCarlC, hartz2017isomorphism, McCarthy_2017} and then consider the spaces as originally defined on such a domain.
We also show that a maximal domain always exists (Proposition \ref{Prop : exsit}), although there is no uniqueness (Example \ref{Ex : 1}).

Usually, weighted Hardy spaces are considered on the Euclidean unit ball $\bB_d$; see \cite{CowenCarlC,hartz2017isomorphism, Gilad}. We allow some flexibility and consider those spaces on their natural maximal domains, which may be smaller or larger than the unit ball. 
It turns out that every weighted Hardy space corresponds to a unique function, known as the generating function, which is closely related to its natural maximal domain; see Lemma \ref{lem : domof}.
We classify all algebraically consistent weighted Hardy spaces (Theorem \ref{thm : graet}) and deduce an analog to \cite[Corollary 3.4]{MR3687947}.
\begin{corollary}
    Let $\cH$ be a weighted Hardy space with the scalar Pick property and let $\rho$ be a bounded non-zero functional on $\cH$. Then the following are equivalent:
    \begin{enumerate}
        \item $\rho(\phi g)=\rho(\phi)\rho(g)$ for all $\phi\in\Mult(\cH)$ and all $g\in\cH$.
        \item $\rho(fg)=\rho(f)\rho(g)$ whenever $f,g\in\cH$ such that $fg\in\cH$.
    \end{enumerate}
\end{corollary}
Motivated by the second paragraph, we aim to classify the weighted Hardy spaces up to an RKHS isomorphism, thereby complementing the first result in Ofek and Sofer's work \cite{Gilad}.
\begin{theorem}\cite[Theorem 1.2]{Gilad}\label{thm : gilad}
    Let $\cH_1$ and $\cH_2$ be weighted Hardy spaces in one variable.  
    Supposed that $G_1(t)=\sum_{n=0}^\infty a_nt^n$ is the generating function of $\cH_1$  and $G_2(t)=\sum_{n=0}^\infty b_nt^n$ is the generating function of $\cH_2$.
    Assume further that both $G_1$ and $G_2$ are holomorphic on $\bB_1$.
    Then $\cH_1$ and $\cH_2$ are isomorphic via an RKHS isomorphism $T$ such that $T^\ast=\alpha C_\varphi$ with $\varphi(\bB_1)=\bB_1$, if and only if there exist $c, C\in (0,\infty)$ such that $c\le\frac{a_n}{b_n}\le C$ for all $n$.

    In this case, $T$ is an isometric RKHS isomorphism if and only if $a_n=b_n$ for all $n$.
\end{theorem}

We generalize (Theorems \ref{thm : big1}, \ref{thm : big2}, \ref{thm : big3}, \ref{thm : big 4}) this result in two principal directions:
\begin{itemize}
    \item Considering weighted composition operators rather than composition operators.
    \item Considering spaces defined on $r\bB_1$ with $r\in(0,\infty]$, or on $r\bB_d$ with $r\in(0,\infty)$ for arbitrary $d\ge 1$.
\end{itemize}
Broadly speaking, we establish a comparable relation between the weights of the generating functions when there is an RKHS isomorphism between spaces defined on balls with the same dimension.

At the end of the paper, we introduce the reproducing kernel Banach-Mazur distance, which originally appeared in \cite{OFEK}.
One can show that it always induces a distance function for finite-dimensional spaces with the same dimension.
Although it is not the case with general weighted Hardy spaces, we use our results to establish a valid distance function on certain equivalence classes; see Corollary \ref{cor:end}. 
Moreover, using methods developed in the paper, we find that spaces are close in the induced distance if and only if the weights of the generating functions are close; see Theorem \ref{thm:est}.

Most of the proofs in this paper rely on basic results from the theory of holomorphic functions, which can be found in \cite{curry2025tasty,rudin1980function}.
For more background about Hilbert function spaces, we refer the reader to \cite{Jim, Paulson}.
 
\section{Preliminaries and notation}
\subsection{Function spaces and kernels}
A Banach space $\cH$ of functions on a set $X$ is called \emph{a functional Banach space on $X$} if, for every $x \in X$, the point evaluation functional $f\mapsto f(x)$ is bounded. 
$\cH$ is called \emph{separating points} if $f(x)=f(y)$ for all functions in the space implies $x=y$.

When $\cH$ is a Hilbert space, we call it a \emph{Hilbert function space}, or \emph{reproducing kernel Hilbert space} (RKHS for short).
The Riesz representation theorem implies that for every $x\in X$ there exists $k_x\in \cH $ such that $f(x)=\inp{f,k_x}$ for all $f\in \cH$.
The \emph{reproducing kernel associated} to $\cH$ is the function $K : X \times X \to \C$ defined by $K(x,y)=k_y(x)=\langle k_y,k_x\rangle$.
We also have the relation $\norm{k_x}^2=\inp{k_x,k_x}=K(x,x)$.

If $X$ is a nonempty set and $K: X \times X \to \C$ is a function,  then $K$ is a called \emph{kernel on $X$} if for every finite subset ${x_1,...,x_n}$ of $X$, the matrix $[K(x_i,x_j)]$ is positive semi-definite. 

The reproducing kernel associated with a Hilbert function space is automatically a kernel on $X$.  
Moore \cite{moore1939generalanalysis2} showed that, just as every Hilbert function space on $X$ gives rise to a kernel on $X$, one can also start with a kernel function and construct a Hilbert function space with that function as its reproducing kernel.

\begin{proposition}
    Let $\cH$ be a Hilbert function space with reproducing kernel $K$, and let $m$ be a natural number. 
    The following are equivalent:
    \begin{enumerate}
        \item For any set of distinct point $x_1,...,x_m$, the vectors $k_{x_1},...,k_{x_m}$ are linearly independent;
        \item For any set of distinct point $x_1,...,x_m$, the matrix $[K(x_i,x_j)]$ is invertible;
        \item For any set of distinct point $x_1,...,x_m$, there exist functions, $p_1,...,p_m\in \cH$ satisfying 
        $$p_i(x_j)=\begin{cases} 1 & i=j \\ 0 &i\neq j \end{cases}.$$
    \end{enumerate}
\end{proposition}
\begin{proof}
    $(1)\Rightarrow (2):$ Assume $[K(x_i,x_j)]$ is not invertible, and let $v=(\alpha_1,...,\alpha_m)^t$ such that $[K(x_i,x_j)]v=0$. Set $f=\sum \alpha_jk_{x_j}$ then
    $$\norm{f}^2=\sum_{i,j}\ol{\alpha_i}\alpha_j\inp{k_{x_j},k_{x_i}}_\cH=\sum_{i,j}\ol{\alpha_i}\alpha_jK(x_i,x_j)=\inp{[K(x_i,x_j)]v,v}_{\C^m}=0.$$
    Hence, $f=0$ and $k_{x_1},...,k_{x_m}$ are not linearly independent.

    $(2)\Rightarrow(3):$ Assume $[K(x_i,x_j)]$ is invertible, and let $v=(y_1,...,y_m)^t$. 
    We will show that there exists a function $f\in\cH$ such that $f(x_i)=y_i$ for $i=1,\dots ,m$.
    Let $w=(\alpha_1,...,\alpha_m)^t$ such that $[K(x_i,x_j)]w=v$. 
    Set $f=\sum \alpha_jk_{x_j}$ then for $i=1,...,m$ we have
    $$f(x_i)=\inp{f,k_{x_i}}_\cH=\sum_j \alpha_j K(x_j,x_i)=y_i.$$
    And $(3)$ follows.

    $(3)\Rightarrow(1):$ If $\sum_j \alpha_jk_{x_j}=0$ then for all $i=1,...,m$ we get
    $$0=\inp{p_i,\sum_j \alpha_jk_{x_j}}_\cH=\sum_j \ol{\alpha_j}\inp{p_i,k_{x_j}}_\cH=\sum_j \ol{\alpha_j}p_i(x_j)=\ol{\alpha_j}.$$
    So, $k_{x_1},...,k_{x_m}$ are linearly independent, and the proof is complete.
\end{proof}
We say that $\cH$ is \emph{$m$-interpolating} if any of the equivalent conditions of the above proposition hold.   
$\cH$ is called \emph{fully interpolating}, if it $m$-interpolating for all natural $m$ \cite[Definition 3.7]{Paulson}. 
We say that $\cH$ is \emph{normalized at the point} $x_0\in X$, if $K(x,x_0)=1$ for all $x\in X$. 
It is immediate that, for normalized spaces, the terms $2$-interpolating and separating point are equivalent.

We notice that a function $f \in \cH$ is orthogonal to the span of $\set{k_x:x\in X}$ if and only if $f(x)=\inp{f,k_x}=0$ for every $x \in X$, which is if and only if
$f = 0$. 
Hence, the set $\spa\set{k_x:x\in X}$ is dense in $\cH$.
Another fundamental result is that weak convergence implies pointwise convergence, i.e., if $f_n\xrightarrow[]{w}f$ then for each $x\in X$ we have
$$\lim_n f_n(x)=\lim_n\inp{f_n,k_x}=\inp{f,k_x}=f(x).$$

\subsection{Multipliers and composition operators}
Let $\cB_i$ be a functional Banach space defined on $X_i$, $i=1,2$.
A \emph{multiplier} of $\cB_1$ into $\cB_2$ is a function $\lambda$ with the property that $\lambda f\in\cB_2$ for all $f\in\cB_1$. We let $\on{Mult}\pare{\cB_1,\cB_2}$ denote the set of all multipliers of $\cB_1$ into $\cB_2$ and $\Mult(\cB)=\Mult(\cB,\cB)$ for short.
For a multiplier $\lambda\in\Mult(\cB_1,\cB_2)$, we let 
$M_\lambda:\cB_1\to\cB_2$ denote the linear map $f\mapsto \lambda f$, 
and call such an operator \emph{the multiplication operator by $\lambda$}.
Notice that when $\cB_1$ separates points, there is a one-to-one correspondence between $\Mult(\cB_1,\cB_2)$ and the set of multiplication operators.
In particular, if $\cH$ is normalized, then $\Mult(\cH)\subset \cH$, and the inclusion is contractive.

Another important class of operators consists of \emph{composition operators} $C_\varphi$, which are defined by $f\mapsto f\circ\varphi$. 
For more details, we refer the reader to \cite[Chapter 5]{Paulson} and \cite{Jim, CowenCarlC}.

\subsection{The Pick property}
 A Hilbert function space $\cH$ with reproducing kernel $K$ on $X$ has \emph{the scalar Pick property} if whenever $x_1,...,x_m\in X$ are distinct points and $y_1,\dots,y_m\in\C$ are given such that the Pick matrix
    $$[(1-y_i\ol{y_j})K(x_i,x_j)]$$
is positive semi-definite, there exists a multiplier $\lambda$ in the closed unit ball of $\Mult(\cH)$ such that $\lambda(x_i) = y_i$ for $i = 1,..., m$.

When the analogous result for matrix-valued interpolation holds, $\cH$ is said to be a complete Pick space; for more details, see \cite{Jim}.

\subsection{$\C^d$ notations}
For $d\in\N$ and $r\in(0,\infty]$, we denote the open ball $\set{z\in\C^d:\av{z}<r}$ by $r\bB_d$ and the closed ball $\set{z\in\C^d:\av{z}\le r}$ by $\ol{r\bB_d}$, of course $r\bB_d=\ol{r\bB_d}=\C^d$ when $r=\infty$.
When $d=1$, we denote by $\D$ the unit disc $\bB_1$.
We denote by $\mathcal{U}(d)$ the set of all unitary maps on $\C^d$.
For a ball $X\subseteq \C^d$ (closed or open), $\on{Aut}(X)$ denotes the group of the bi-holomorphic maps of $\on{int}(X)$, and $H^\infty(X)$ denotes the holomorphic functions on $\on{int}(X)$ that are also continuous and uniformly bounded on $X$.
Every unitary on $\C^d$ yields an element of $\on{Aut}(\bB_d)$, and by Cartan's uniqueness theorem, every automorphism fixing the origin is unitary \cite[Corollary 1.5.2]{curry2025tasty}.
We adopt the following multi-index notation: if $\alpha=(\alpha_1,\dots,\alpha_d)$ is a multi-index of non-negative integers, then $\av{\alpha}:=\alpha_1+\dots+\alpha_d$ and $\alpha!:=\alpha_1!\cdots\alpha_d!$ and for $z=(z_1,\dots,z_d)\in\C^d$, $z^\alpha:=z_1^{\alpha_1}\cdots z_d^{\alpha_d}$ and $\ol{z}^\alpha:=\ol{z_1}^{\alpha_1}\cdots \ol{z_d}^{\alpha_d}$.

\section{Domains}
In this section, $\cB$ denotes a reflexive functional Banach space that separates points of $X$.

When thinking of $\cB$ as just a set of functions, we can always consider functions defined only on a subset of $X$. However, when we do so, we sometimes change the structure of the space, for example, when the restriction map is not injective.
This section discusses the domains worth considering. 
We must warn the reader that the following definitions have similar but subtly differing variants in the literature; see \cite{MR3687947, CowenCarlC, hartz2017isomorphism, McCarthy_2017,  MR4225497}.
We adopt the term maximal domain as in \cite{MR3687947, MR4225497}, which seems to have become conventional.

\begin{definition}
    A nonzero bounded linear functional $\rho$ on $\cB$ is called \emph{a multiplicative functional} if $\rho(fg) = \rho(f)\rho(g)$ whenever $f, g$, and the product $fg$ are in $\cB$.
    We say that \emph{$X$ is a maximal domain for $\cB$} if for every multiplicative functional $\rho$ on $\cB$, there exists $x\in X$ such that $\rho(f)=f(x)$ for all $f\in \cB$.
\end{definition}
Observing that for every point in $X$, the associated evaluation functional is a multiplicative functional, we obtain:
\begin{proposition}\cite[Proposition 7]{McCarthy_2017}\label{Prop : exsit}
    Let $\cB$ be as above. 
    There exists a functional Banach space $\widehat{\cB}$ with a maximal domain $\widehat{X}$ and an injective map $\varphi: X\to \widehat{X}$ such that $C_\varphi:\widehat{\cB}\to \cB$ is an isometric isomorphism. Moreover, when $\cB$ is a Hilbert space, $\widehat{\cB}$ is also a Hilbert space.
\end{proposition}
\begin{proof}
    Define $\widehat{X}=\set{\rho\in \cB^*: \rho \text{ is a multiplicative functional}}$ and $\varphi:X\to \widehat{X}$ by $\varphi(x)=\on{ev}_x$.
    For $f\in\cB$ we denote by $\hat{f}$ the corresponding functional in $\cB^{**}$.
    Consider the space $$\widehat{\cB}=\set{\hat{f}:\hat{X}\to \C:f\in\cB}\subseteq \cB^{**}$$
    as the functional Banach space defined on $\widehat{X}$. 
    Notice that $\hat{f}(\varphi(x))=f(x)$ and the map $f\mapsto \hat{f}$ is isometric isomorphism which its inverse is $C_\varphi$. 
    Additionally, if $\hat{f},\hat{g}\in \widehat{\cB}$ are such that $\hat{f}\hat{g}\in\widehat{\cB}$ then $fg\in \cB$ and $\widehat{fg}=\hat{f}\hat{g}$, so by the reflexive assumption $\widehat{X}$ is a maximal domain for $\widehat{\cB}$.
\end{proof}
Some readers might already understand the worth of maximal domains in the question of whether the inverse of a weighted composition operator is another weighted composition operator.

In other words, $X$ is a maximal domain for $\cB$ when the map $\varphi: X\to \hat{X}$ is a bijection. 
We also understand $\widehat{X}$ as a maximal extension of $X$, which allows us to think of $\cB$ as a functional Banach space originally defined on $\widehat{X}$.
Notice that the construction of $\widehat{X}$ endows it with some topologies. We would like to identify $\widehat{X}$ with a concrete subset of $\C^d$:
\begin{proposition}\cite[Proposition 8]{McCarthy_2017}
    Let $\cB$ be as above and suppose that the algebra generated by $\phi_1,\dots,\phi_d$ is contained in and dense in $\cB$. Then $\widehat{X}$ can be identify as 
    $$\widehat{X}=\set{(\rho(\phi_1),\dots,\rho(\phi_d)):\rho \text{ is a multiplicative functional}}$$
\end{proposition}
If one of the $\phi_i$’s equals the constant function $\textbf{1}$, then it can be omitted in the above construction since $\rho(f)=\rho(\textbf{1})\rho(f)$ force $\rho(\textbf{1})=1$. 
Thus, if $\cB$ contains the algebra of polynomials with $d$ variables as a dense subspace, we can identify a subset of $\C^d$ as a maximal domain for it.

\begin{example}\label{Ex : 1}
    Although we show the existence of a maximal domain, it turns out it is not unique in some sense:
    
    Let $\cH$ and $\cK$ be the spaces of all polynomials of degree at most $2$, with the set $\set{1,z,z^2}$ forming a complete orthonormal basis.
    Consider $\cH$ as a space of functions on $\C$ and $\cK$ as a space of functions on $X=\set{0,1,2}$.
    Notice that the map $p\mapsto p|_X$ is clearly a bijection and is induced by a composition operator $C_\iota:\cH\to\cK$, where $\iota: X\to \C$ is the inclusion map.
    It is easy to see that every multiplicative functional on $\cH$ arises from pointwise evaluation and vice versa.
    If $\rho$ is a multiplicative functional on $\cK$, set $t=\rho(z)$ then $\rho(az^2+bz+c)=at^2+bt+c$. We show that $t\in X$.
    For $f(z)=z,g(z)=z^2$ and $h(z)=3z^2-2z$, we see that $f|_Xg|_X=h|_X$ (while $fg\notin \cH$), so
    $$t^3=\rho(f|_X)\rho(g|_X)=\rho(h|_X)=3t^2-2t.$$
    Solving for $t$, we obtain that $t\in X$. Therefore, $X$ is a maximal domain for $\cK$.
    
    Notice that both $\cH$ and $\cK$ are $2$-interpolating, and using this idea, one can find an example of a set of functions that define isometrically isomorphic spaces that are $m$-interpolating with different maximal domains. 

    At first glance, we wish for those spaces to be isomorphic via an RKHS isomorphism. 
    However, when looking carefully, we see that these spaces have different structures. 
    This example demonstrates that different maximal domains can imply different multiplication structures (here, $\cK$ is an algebra) and highlights a problem with the terminology of maximal domain. 
\end{example}
     
    Example \ref{Ex : 1} shows that a subset $S\subset X$ should be called faithful if the map $f\mapsto f|_S$ is injective and $f|_Sg|_S=h|_S$ if and only if $h=fg$, that is, the map $f\mapsto f|_S$ is a multiplicative unitary in the sense of \cite[Definition 12]{McCarthy_2017}.
    We already saw in Proposition \ref{Prop : exsit} that $X\subseteq\widehat{X}$ should be a faithful subset, and in fact, a modification of \cite[Proposition 17]{McCarthy_2017} shows that these are equivalent: 
\begin{proposition}
    Assume that $X$ is a maximal domain for $\cB$. 
    For $S\subset X$, define $$\cB_S=\set{f|_S: f\in \cB}.$$
    The following are equivalent 
        \begin{enumerate}
            \item $\widehat{S}=\widehat{X}$ and $\widehat{\cB_S}=\widehat{\cB}$.
            \item The map $f\mapsto f|_S$ is isometric isomorphism and $f|_Sg|_S=h|_S$ if and only if $h=fg$.
        \end{enumerate}
\end{proposition}
\begin{proof}
     $(1) \Rightarrow (2)$ is Proposition \ref{Prop : exsit}. Assume $(2)$ holds. Note that $\cB_S^*$ can be identified isometrically with $\cB^*$, so $\widehat{S}\subseteq \widehat{X}$.
     Moreover, if $\rho\in \widehat{X}$ and $f|_Sg|_S=h|_S$, then
     $$\rho(h|_S)=\rho(h)=\rho(f)\rho(g)=\rho(f|_S)\rho(g|_S).$$
     Hence, $\rho\in \widehat{S}$ and $(1)$ holds.
\end{proof}
\begin{definition}
    Assume that $X$ is a maximal domain for $\cB$. 
    A subset $S\subset X$ is called \emph{a faithful subset} if one of the equivalent conditions of the above proposition holds.  
\end{definition}
    In other words, the maximal domain defines the multiplication in the space. With this definition, we now understand the terminology of maximal domains.

    \medskip
Now assume that $\textbf{1}\in \cB$ and $X$ is a maximal domain for $\cB$.
\begin{definition}
    A nonzero bounded linear functional $\rho$ on $\cB$ is called \emph{a partially multiplicative functional} if $\rho(\phi g) = \rho(\phi )\rho(g)$ whenever $\phi \in \Mult(\cB)$ and $g\in \cB$. 
    We say that $\cB$ is \emph{algebraically consistent on the maximal domain $X$} if every partially multiplicative functional $\rho$ on $\cB$ is a multiplicative functional.
\end{definition}
We can find some examples of algebraically consistent spaces (with our definitions) in \cite[Corollary 3.4]{MR3687947} and \cite[Theorem 1.7]{MR4225497}.
Theorem \ref{thm : graet} below classifies the algebraically consistent weighted Hardy spaces on the natural maximal domain.

From now on, although we can formulate some of the results for functional Banach spaces, we consider only Hilbert function spaces.

\section{ Weighted composition operators}
In this section, $\cH_i$ denotes a Hilbert function space defined on $X_i$ with reproducing kernel $K_i$, $i=1,2$. We address the cases where the space is defined on a maximal domain.

\begin{definition}
    A \emph{weighted composition operator} $W_{\lambda,\varphi}:\cH_2\to \cH_1$ is an operator defined by $f\mapsto \lambda (f\circ\varphi)$, where $\varphi: X_1\to X_2$ is called \emph{the composition symbol of $W_{\lambda,\varphi}$} and $\lambda:X_1\to \C$ is called \emph{the multiplier symbol of $W_{\lambda,\varphi}$}. 
\end{definition}
The following theorem extends \cite[Section 5]{Paulson} and shows that a well-defined weighted composition operator is bounded. 
\begin{proposition}\label{Thm : bound}
    Let $\cH_1$ and $\cH_2$ be as above.
    For $\varphi: X_1\to X_2$ and $\lambda:X_1\to \C$ the following are equivalent:
    \begin{enumerate}
        \item $\set{\lambda \cdot f\circ\varphi:f\in\cH_2 }\subseteq \cH_1$.
        \item $W_{\lambda,\varphi}:\cH_2\to \cH_1$ is a bounded operator. 
        \item there exists a constant $c\ge0$, such that $c^2K_1(x,y)-\ol{\lambda(y)}K_2(\varphi(x),\varphi(y))\lambda(x)$ is kernel function on $X_1$. 
    \end{enumerate}
    Moreover, in these cases, $\norm{W_{\lambda,\varphi}}$ is the least constant, $c$, satisfying the condition in (3).
\end{proposition}
\begin{proof}
    To keep this paper concise, we refer the reader to \cite[Theorems 3.1 and 3.3]{kumari2025composition}.
\end{proof}
From now on, to simplify writing, we always assume that $W_{\lambda,\varphi}:\cH_2\to \cH_1$ is a bounded weighted composition operator.
The following lemma demonstrates the relation between a weighted composition operator and its adjoint.
\begin{lemma}\label{lem : formw}
    Assume that $\cH_1$ and $\cH_2$ are $2$-interpolating and $T:\cH_2\to\cH_1$ is a bounded linear operator. 
    The following are equivalent:
    \begin{enumerate}
        \item $T$ is a weighted composition operator.
        \item The set $\set{k^1_x:x\in X_1}$ is sent by $T^\ast$ to a subset of $\set{\alpha k^2_y:\alpha\in \C, y\in X_2}$.
    \end{enumerate}  
    In these cases, $T= W_{\lambda,\varphi}$ where $\lambda,\varphi$ and $T$ are related by $W_{\lambda,\varphi}^\ast (k_x^1)=\ol{\lambda(x)}k^2_{\varphi(x)}$.
\end{lemma}
\begin{proof}
    $(1)\Rightarrow(2):$ If $T= W_{\lambda,\varphi}$ is a weighted composition operator, then for any $f\in \cH_2$ and $x\in X_1$ we have 
    \begin{align*}
    \inp{f,W_{\lambda,\varphi}^\ast (k_x^1)}_{\cH_2}
    & =\inp{W_{\lambda,\varphi}(f), k_x^1}_{\cH_1}
    =\inp{\lambda f\circ\varphi,k_x^1}_{\cH_1} \\
    & =\lambda(x)f(\varphi(x))
    =\inp{f,\ol{\lambda(x)}k^2_{\varphi(x)}}_{\cH_2}.    
    \end{align*}
    Thus, $W_{\lambda,\varphi}^\ast (k_x^1)=\ol{\lambda(x)}k^2_{\varphi(x)}$ and $\set{k^1_x:x\in X_1}$ is sent to a subset of $\set{\alpha k^2_y:\alpha\in \C, y\in X_2}$.

    $(2)\Rightarrow(1):$ Conversely, if $\set{k^1_x:x\in X_1}$ is sent by $T^\ast$ to a subset of $\set{\alpha k^2_y:\alpha\in \C, y\in X_2}$, let $\lambda$ and $\varphi$ be define by, $T^\ast(k^1_x)=\ol{\lambda(x)}k^2_{\varphi(x)}$ for every $x\in X_1$.
    Those maps are well-defined since $\cH_i$ are  $2$-interpolating and we can define arbitrary $\varphi(x)$ when $T^\ast(k^1_x)=0$. 
    Then
    $$T(f)(x)=\inp{T(f),k^1_x}_{\cH_1}=\inp{f,T^\ast(k^1_x)}_{\cH_2}=\inp{f,\ol{\lambda(x)}k^2_{\varphi(x)}}_{\cH_2}=\lambda(x)f(\varphi(x))$$
    so, $T=W_{\lambda,\varphi}$.
\end{proof}
From this lemma, for $W_{\lambda,\varphi}:\cH_2\to \cH_1$ we can infer that 
$$\av{\lambda(x)}^2K_2(\varphi(x),\varphi(x))=\norm{W^\ast_{\lambda,\varphi}(k^1_x)}_{\cH_2}^2\le \norm{W^\ast_{\lambda,\varphi}}^2\norm{k_x^1}_{\cH_1}^2=\norm{W^\ast_{\lambda,\varphi}}^2 K_1(x,x)$$
for all $x\in X_1$.
In particular, when $W_{\lambda,\varphi}:\cH_2\to\cH_1$ is the multiplier operator $M_\lambda$, i.e. $\cH_1$ and $\cH_2$ define on $X$ and $\varphi=\on{id}_X$, we obtain that $\av{\lambda(x)}\le \norm{M_\lambda}$ for all $x\in X$. 
Hence, if $\lambda$ is holomorphic function we get that $\lambda\in H^\infty(X)$.
Notice that $C_\varphi$ is another special case of a weighted composition operator; here $\lambda\equiv\textbf{1}$.
The lemma also yields a known description of composition operators in terms of the property of maximal domains.
\begin{corollary}
    Assume that $\cH_1$ and $\cH_2$ are $2$-interpolating and $T:\cH_2\to\cH_1$ is a bounded linear operator.
    Suppose also that $X_2$ is a maximal domain for $\cH_2$ and $T^{\ast}(k^1_x)\neq 0$ for all $x\in X_1$. 
    The following are equivalent:
    \begin{enumerate}
        \item $T$ is a composition operator.
        \item $T(fg)=T(f)T(g)$ whenever $f,g\in\cH_2$ such that $fg\in\cH_2$.
    \end{enumerate}
\end{corollary}
\begin{proof}
    Clearly $(1)\Rightarrow(2)$.
    
    For $(2)\Rightarrow (1)$ we show that there is a map $\varphi:X_1\to X_2$ such that $T=C_\varphi$. Fix $x\in X_1$. Since $X_2$ is a maximal domain for $\cH_2$ by Lemma \ref{lem : formw}, it is sufficient to show that $k=T^{\ast}(k^1_x)$ induces a multiplicative functional on $\cH_2$.
    For $f,g,fg\in\cH_2$ we have
    \begin{align*}
        \inp{fg,k}_{\cH_2} &=\inp{T(fg),k^1_x}_{\cH_1}
        =\inp{T(f)T(g),k^1_x}_{\cH_1} \\
        &= \inp{T(f),k^1_x}_{\cH_1}\inp{T(g),k^1_x}_{\cH_1} \\
        & =\inp{f,k}_{\cH_2}\inp{g,k}_{\cH_2}.
    \end{align*}
    Since $k\neq 0$, it induces a nonzero functional on $\cH_2$ and, in particular, a multiplicative functional on $\cH_2$, which completes the proof.
\end{proof}
\begin{remark}
    Where $\cH_2$ is not $1$-interpolating there is $y\in X_2$ such that $f(y)=0$ for all $f\in\cH_2$. Hence, for $\varphi\equiv y$, $C_\varphi:\cH_2\to\cH_1$ is the zero map.
\end{remark}
\medskip
\begin{definition}
    \emph{An RKHS isomorphism} from $\cH_1$ to $\cH_2$ is a bijective bounded linear map $T : \cH_1 \to \cH_2$ defined by $T(k_x^1)=\overline{\lambda(x)}k_{\varphi(x)}^2$ for all $x\in X_1$, where $\lambda: X_1\to\C$ is a non-vanishing function and $\varphi: X_1\to X_2$ is a bijection.
\end{definition}
As we already mentioned, by Lemma \ref{lem : formw}, the adjoint map of an RKHS isomorphism is a weighted composition operator.
We also know that an RKHS isomorphism induces an isomorphism between the multiplier algebras.
The following proposition relates to \cite[Proposition 4.4]{OFEK}.   
\begin{proposition}\label{prop : isoalg}
    Assume that $\cH_1$ and $\cH_2$ are $2$-interpolating.
    If $W_{\lambda,\varphi}:\cH_2\to\cH_1$ is invertible such that $\lambda:X_1\to \C$ never vanishes and $\varphi:X_1\to X_2$ is a bijection, then $C_\varphi:\Mult(\cH_2)\to\Mult(\cH_1)$ is a norm continues isomorphism.
\end{proposition}
\begin{remark}
A few words before the proof
\begin{enumerate}
    \item One can show that we can ask for the spaces to be just separating points.
    \item Recall that the assumption that $\cH_i$ is separating points ensures a one-to-one correspondence between $\Mult(\cH_i)$ and the set of multiplication operators.
    \item We show soon, Lemma \ref{lem : inj}, that when $W_{\lambda,\varphi}$ is invertible,  automatically $\lambda$ never vanishes and $\varphi$ is injective.
\end{enumerate}
\end{remark}
\begin{proof}
    Note that if $\Phi\in\Mult(\cH_2)$ then $W_{\lambda,\varphi}M_\Phi W_{\lambda,\varphi}^{-1}$ is the multiplication operator by $\Phi\circ\varphi$.
    Indeed, if $W_{\lambda,\varphi}(F)=f\in \cH_1$ then
    $$W_{\lambda,\varphi}M_\Phi W_{\lambda,\varphi}^{-1}(f)=W_{\lambda,\varphi}(\Phi F)=\lambda \Phi\circ\varphi F\circ\varphi=\Phi\circ\varphi f.$$
    Hence, the map is well defined.
    It is clear that the map is injective; it remains to show that it is onto.
    To this end, for a nonzero multiplier $\phi\in \Mult(\cH_1)$ we define a map $T:\cH_2\to \cH_2$ by $T=W_{\lambda,\varphi}^{-1}M_\phi W_{\lambda,\varphi}$. 
    We aim to show that $T$ is a multiplication operator by $\Phi=\phi \circ\varphi^{-1}$.
    For $\varphi(x)\in X_2$ we have
     \begin{align*}
        \ol{\lambda(x)}T^\ast(k^2_{\varphi(x)}) 
        & =W_{\lambda,\varphi}^\ast M_\phi ^\ast \pare{W_{\lambda,\varphi}^\ast}^{-1}(\ol{\lambda(x)}k^2_{\varphi(x)})
        =W_{\lambda,\varphi}^\ast M_\phi ^\ast\pare{k^1_x} \\
        & =W_{\lambda,\varphi}^\ast\pare{\ol{\phi(x)}k^1_x} =\ol{\lambda(x)}\ol{\phi(x)}k^2_{\varphi(x)}
        =\ol{\lambda(x)}\pare{\ol{\Phi(\varphi(x))}k^2_{\varphi(x)}} 
    \end{align*}
    since $\lambda(x)\neq 0$ and $\varphi$ in bijection, by Lemma \ref{lem : formw} $T=M_\Phi$ and maps onto $M_\phi$ which completes the proof.
\end{proof}

Suppose $W_{\lambda,\varphi}:\cH_2\to \cH_1$ is an invertible operator.
The first problem we will address is: 
\begin{quest}
    Under which conditions does $W_{\lambda, \varphi}$ become the adjoint of an RKHS isomorphism? 
    Namely, under what conditions does $\lambda$ not vanish and $\varphi$ is a bijection?
\end{quest}

In other words, we seek to understand when $W_{\lambda,\varphi}^{-1}$ is also a weighted composition operator. 
We use the terminology $W_{\lambda,\varphi}$ is invertible to emphasize that $W_{\lambda,\varphi}$ need not be the adjoint of an RKHS isomorphism. 

We want to gain a better understanding of the symbols $\lambda$ and $\varphi$ for which the weighted composition operator $W_{\lambda,\varphi}$ is invertible.
\begin{lemma}\label{lem : inj}
     Assume that $\cH_1$ and $\cH_2$ are $2$-interpolating.    
     If $W_{\lambda,\varphi}^\ast:\cH_1\to\cH_2$ is injective, then: 
     \begin{enumerate}
         \item $\lambda:X_1\to \C$ does not vanish. 
         \item $\varphi: X_1\to X_2$ is an injective map. 
     \end{enumerate} 
     Moreover, when $\cH_1$ consists of functions that are holomorphic on the interior of $X_1$ and continuous on $X_1\subseteq \C^{d_1}$, and $\cH_2$ contains the constant function and the coordinate functions on $X_2\subseteq\C^{d_2}$, then $\lambda$ and $\varphi$ are holomorphic on the interior of $X_1$ and continuous on $X_1$. 
\end{lemma}
\begin{proof}
    First, by the $2$-interpolating property, we get $k^i_x\neq 0$ for any $x\in X_i$,  $i=1,2$.
    Thus,
    $$W_{\lambda,\varphi}^\ast(k_x^1)=\ol{\lambda(x)}k_{\varphi(x)}^2\neq 0$$
    for any $x\in X_1$. This shows that $\lambda$ does not vanish. 
    
    Since $k_x^1$ and $k_y^1$ are linearly independent for all distinct $x,y\in X_1$, $\varphi$ is injective. Indeed, if $\varphi(x)=\varphi(y)$ then    $$\ol{\lambda(y)}W_{\lambda,\varphi}^\ast(k_x^1)=\ol{\lambda(y)}\ol{\lambda(x)}k_{\varphi(x)}^2=\ol{\lambda(x)}W_{\lambda,\varphi}^\ast(k_y^1).$$
    Hence, $\ol{\lambda(y)}k_x^1=\ol{\lambda(x)}k_y^1$, so $x=y$ and $\varphi$ is injective.
    
    Notice that $W_{\lambda,\varphi}(\textbf{1})=\lambda\in\cH_1$, so $\lambda$ is holomorphic on the interior of $X_1$ and continuous on $X_1$. 
    Let $\zeta \in \C^{d_2}$ be with finite support. The assumption that $\cH_2$ contains the coordinate functions implies that $\inp{\cdot,\zeta}_{d_2} \in \cH_2$. 
    Hence,
    $$ W_{\lambda,\varphi}\pare{\inp{\cdot,\zeta}_{d_2}}=\lambda\inp{\varphi(\cdot),\zeta}_{d_2}\in\cH_1. $$
    Since $\lambda$ does not vanish, $\varphi$ is holomorphic in the interior of $X_1$ and continuous on $X_1$.
\end{proof}

Although most work on weighted composition operators does not require considering spaces with their maximal domains, we understand that this is crucial for this problem.
In Proposition \ref{Prop : exsit} and Example \ref{Ex : 1}, we saw examples of unitary composition operators with composition symbols that are not necessarily onto.
We illustrate this issue with another example:
\begin{example}\label{ex:Orr}
    This example is taken from \cite[Section 5]{McCarthy_2017}.
    Let $\{b_k\}^\infty_{k=1}$ be a sequence of positive numbers such that $\sum_{k=1}^\infty b_k^2=1$.
    Set $\bH_0=\set{z\in\C :\Re{z}>0}$, $\bB_\infty=\set{z\in\ell^2:\norm{z}<1}$ and $\psi:\bH_0\to\bB_d$ to be defined by $\psi(s)=(b_1p_1^{-s},b_2p_2^{-s},b_3p_3^{-s},\dots)$, where $p_k$ denotes the $k$-th prime.
    Let $H^2_\infty$ denote the Hilbert function space on $\bB_\infty$ with $\mathfrak{a}(z,w)=\frac{1}{1-\inp{z,w}_{\ell^2}}$ as its reproducing kernel (the infinite-dimensional Drury--Arveson space).
        
    We define a kernel on $\bH_0$ by $K(s,u)=\mathfrak{a}(\psi(s),\psi(u))$ and we denote by $\cH$ the Hilbert function space with $K$ as its kernel.
    McCarthy and Shailt have showed, see \cite[Theorem 31]{McCarthy_2017}, that the map $T:k_s\mapsto\mathfrak{a}_{\psi(s)}$ extends to a unitary map from $\cH$ onto $H^2_\infty$.
    The adjoint of $T$ is a composition operator given by $T^\ast f = f \circ \psi$.
    Assume for the sake of reaching a contradiction that $\lambda:\bB_\infty\to\C$ and $\varphi:\bB_\infty\to\bH_0$ are such that $W_{\lambda,\varphi} T^\ast=I_{H^2_\infty}$.
    Since $\textbf{1}\in H^2_\infty$ we must have $\lambda=\textbf{1}$ and $W_{\lambda,\varphi}=C_\varphi$.
    Denote by $Z_i$ the $i$-th coordinate map on $\bB_\infty$, then $C_\varphi T^\ast(Z_i)=Z_i$ and
    $z_i=\psi(\varphi(z))_i$ for all $z\in\bB_\infty$ and $i$, i.e. $\psi\circ \varphi=\on{id}_{\bB_\infty}$.
    However, $\psi$ does not map $\bH_0$ onto $\bB_\infty$, so the inverse of $T^\ast$ is not a weighted composition operator when $\cH$ is defined only on $\bH_0$.
    
    The reason is that although $\psi(\bH_0)$ is a faithful subset of $\bB_\infty$, $\bH_0$ is not in a maximal domain of $\cH$.
    In other words, if $X=\widehat{\bH_0}$ is a maximal domain for $\cH$, there exists a map $\Psi:X\to \bB_\infty$ such that $\Psi|_{\bH_0}=\psi$ and $C_\Psi:H^2_\infty\to\cH$ is unitary \cite[Proposition 17]{McCarthy_2017}.

    Unlike Example \ref{Ex : 1} and other natural and basic examples of restricted spaces to uniqueness subset (such as the Hardy space restricted to $\set{\nicefrac{1}{k}}_{k=2}^{\infty}$), there are functions in $\cH$ which can not be extended any further then $\bH_0$ in the complex plain. 
    And yet, $\bH_0$ is not a maximal domain for $\cH$.
\end{example}
Henceforth, when dealing with invertible weighted composition operators, we shall always regard the spaces under consideration as equipped with one of their maximal domains. 
\begin{definition}
    We say that $\cH_1$ is \emph{weakly related to $\cH_2$} if there exists an invertible weighted composition operator $W_{\lambda,\varphi}:\cH_2\to\cH_1$.
    $\cH_1$ and $\cH_2$ are called \emph{strongly related} if $\cH_1$ is weakly related to $\cH_2$ and also $\cH_2$ is weakly related to $\cH_1$, i.e. there exist invertible maps $W_{\lambda,\varphi}:\cH_2\to\cH_1$ and $W_{\mu,\psi}:\cH_1\to\cH_2$, not necessarily that $W_{\lambda,\varphi}^{-1}=W_{\mu,\psi}$.
\end{definition}
The following proposition generalizes \cite[Theorem 1.6]{CowenCarlC} to our setting. 
It is the main tool we use in Section 6 to find sufficient conditions for strongly related spaces (Proposition \ref{cor : sr}) and even weakly related spaces (Proposition \ref{prop : new}) to be isomorphic via an RKHS isomorphism.
 The proof method will be used again in several subsequent arguments.

\begin{proposition}\label{prop : dim}
    Let $\cH_i$ be a $2$-interpolating Hilbert function space with maximal domain $X_i\subseteq \C^{d_i},d_i\in \N$, $i=1,2$.
    Assume that $\on{int}(X_i)$ is connected and dense in $X_i$, 
    and that every function in $\cH_i$ is holomorphic on $\on{int}(X_i)$ and continuous on $X_i$.
    Furthermore, assume that the polynomials form a dense subset of $\cH_1$, $\cH_2$ contains the constants and the coordinate functions, and $d_1\ge d_2$. 
    If $W_{\lambda,\varphi}:\cH_2\to\cH_1$ is invertible, then 
     \begin{enumerate}
         \item $d_1=d_2$,
         \item $\varphi: X_1\to X_2$ is a bijection which is holomorphic on the interior of $X_1$ and continuous on $X_1$,
         \item $W_{\lambda,\varphi}^{-1}=W_{\nicefrac{1}{\lambda\circ\varphi^{-1}},\varphi^{-1}}$. 
     \end{enumerate} 
\end{proposition}
\begin{proof} 
    By Lemma \ref{lem : inj}, $\varphi$ is injective and holomorphic on the interior of $X_1$, and by the assumption that $d_1\ge d_2$, we have $d_1=d_2$. 
    Moreover, by the open mapping theorem \cite[Theorem 15.1.6]{rudin1980function}, $\varphi\big\vert_{\on{int}(X_1)}$ is open.

    Now, for $f,g,fg\in \cH_1$, there are some $F,G, H,K\in\cH_2$ such that
    $$W_{\lambda,\varphi}(F)=f,W_{\lambda,\varphi}(G)=g,W_{\lambda,\varphi}(H)=fg,W_{\lambda,\varphi}(K)=\textbf{1}$$
    i.e. $f=\lambda F\circ\varphi$, $g=\lambda G\circ\varphi$ and $fg=\lambda H\circ\varphi$. 
    Also, $K=\frac{1}{\lambda\circ\varphi^{-1}}$ on $\varphi(X_1)$.
    For any $x\in X_1$ we have
    \begin{gather*}
        \lambda(x)F(\varphi(x))\lambda(x)G(\varphi(x))
        =\langle f,k_x^1\rangle_{\cH_1}\langle g,k_x^1\rangle_{\cH_1} 
        =\langle fg,k_x^1\rangle_{\cH_1}=\lambda(x)H(\varphi(x)).
    \end{gather*}
    Since $\lambda$ never vanishes, we obtain that $FG=KH$ on $\varphi(\on{int}(X_1))$.  Since $FG$ and $KH$ are holomorphic on $\on{int}(X_2)$ and $\varphi(\on{int}(X_1))\subseteq \on{int}(X_2)$ is open, by the identity theorem, see \cite[Theorem 1.2.7]{curry2025tasty}, we get that $FG=KH$ on $\on{int}(X_2)$. 
    As the interior of $X_2$ is dense in $X_2$ and $F, G, H, K$ are continuous on $X_2$, we get $FG=KH$ on $X_2$.
    
    Fix $y\in X_2$ and set $c=K(y)$.
    As $W_{\lambda,\varphi}^\ast$ is onto, there is $h\in\cH_1$ such that $W_{\lambda,\varphi}^\ast(h)=k_y^2$.
    We have
    \begin{align*}
        c\inp{fg,h}_{\cH_1}
        & =c\inp{W_{\lambda,\varphi}(H),h}_{\cH_1}
        =c\inp{H,W_{\lambda,\varphi}^\ast(h)}_{\cH_2}
        =c\inp{H,k_y^2}_{\cH_2} \\
        &=K(y)H(y)=F(y)G(y)
        =\inp{F,k_y^2}_{\cH_2} \inp{G,k_y^2}_{\cH_2} \\ 
        & =\inp{F,W_{\lambda,\varphi}^\ast(h)}_{\cH_2}
        \inp{G,W_{\lambda,\varphi}^\ast(h)}_{\cH_2} 
        =\inp{f,h}_{\cH_1}\inp{g,h}_{\cH_1}.
    \end{align*}        
    Moreover, we have
    $$\inp{\textbf{1},h}_{\cH_1}=\inp{W_{\lambda,\varphi}(K),h}_{\cH_1}=
    \inp{K,k^2_y}_{\cH_2}=K(y)=c.$$
    If $c=0$, then for any multi-index $\alpha=(\alpha_1,\dotsc,\alpha_d)$  we obtain
    $$0=c\inp{z^{2\alpha},h}_{\cH_1} =\inp{z^{\alpha},h}_{\cH_1}\inp{z^{\alpha},h}_{\cH_1}$$ 
    so $\inp{z^{\alpha},h}_{\cH_1}=0$ and we see that $\inp{p,h}_{\cH_1}=0$ for all polynomials $p$. 
    Since the polynomials are dense in $\cH_1$, we get that $h=0$, which is impossible. 
    Thus, $c\neq 0$ and put $k=\frac{1}{\ol{c}}h$, then for all $f,g\in\cH_1$ with $fg\in\cH_1$ we have 
    \begin{align*}
        \inp{fg,k}_{\cH_1}
        &=\inp{fg,\frac{1}{\ol{c}}h}_{\cH_1} =\frac{1}{c}\inp{fg,h}_{\cH_1} \\
        &=\frac{1}{c}\inp{f,h}_{\cH_1}\frac{1}{c}\inp{g,h}_{\cH_1}
        =\inp{f,k}_{\cH_1}\inp{g,k}_{\cH_1}.
    \end{align*}
    By the assumption that $X_1$ is a maximal domain for $\cH_1$, there is some $x\in X_1$ such that $k=k^1_x$ and $$\ol{\lambda(x)}k_{\varphi(x)}^2=W_{\lambda,\varphi}^\ast(k^1_x)=W_{\lambda,\varphi}^\ast(\frac{1}{\ol{c}}h)=\frac{1}{\ol{c}} k_y^2.$$
    Since $\cH_2$ is $2$-interpolating we obtain $\varphi(x)=y$, $K(\varphi(x))=\frac{1}{\lambda(x)}$ and in particular $\varphi$ is onto. 
    
    Finally, we see that 
    \begin{align*}
        W_{\lambda,\varphi}^{-1}(f)(y) = & \inp{W_{\lambda,\varphi}^{-1}(f),k_y^2}_{\cH_2} 
        =\inp{f,\pare{W_{\lambda,\varphi}^\ast}^{-1}(k_y^2)}_{\cH_1} \\ 
        = & \inp{f,\ol{K(y)}k_{\varphi^{-1}(y)}^1}_{\cH_1} 
        = K(y)f\circ\varphi ^{-1}(y) 
    \end{align*}
    for any $f\in\cH_1$ and $y\in X_2$. Thus, as $K=\frac{1}{\lambda\circ \varphi^{-1}}$, we obtain that $\pare{W_{\lambda,\varphi}}^{-1}=W_{\nicefrac{1}{\lambda\circ\varphi^{-1}},\varphi^{-1}}$.
\end{proof}
We have another result with this flavor.
\begin{proposition}\label{prop : comp}
    Let $\cH_i$ be a $2$-interpolating Hilbert function space with maximal domain $X_i$, $i=1,2$.
    If $C_\varphi:\cH_2\to\cH_1$ is invertible such that $\varphi(X_1)\subset X_2$ is faithful subset, then $\varphi$ is a bijection and $C_\varphi^{-1}=C_{\varphi^{-1}}$.
\end{proposition}
\begin{proof}
    By the second part of Lemma \ref{lem : inj}, we see that $\varphi$ is injective. 
    Now, if $f,g,fg\in \cH_1$ there are some $F,G,H\in\cH_2$ such that
    $$C_{\varphi}(F)=f,C_{\varphi}(G)=g,C_{\varphi}(H)=fg$$
    Hence, $H\circ\varphi=fg=F\circ\varphi\cdot G\circ\varphi$, so $H=FG$ on $\varphi(X_1)$. 
    Since we assume that $\varphi(X_1)$ is a faithful subset, we have $H=FG$ on $X_2$.

    The rest of the proof is similar to that of Proposition \ref{prop : dim}, noting that $K\equiv 1$, and is omitted.
\end{proof}
\begin{remark}
When $W_{\lambda,\varphi}:\cH_2\to\cH_1$ is invertible, then since $W_{\lambda,\varphi}^\ast:\cH_2\to\cH_1$ is also invertible, the set $\set{k_{\varphi(x)}^2:x\in X_1}$ spans a dense subspace of $\cH_2$, so any element $f\in\cH_2$ is completely determined by it values on $\varphi(X_1)$, i.e. if there exists $g\in\cH_2$ such that $g=f$ on $\varphi(X_1)$ then $f-g\perp\set{k_{\varphi(x)}^2:x\in X_1}$ so $f=g$ on $X_2$.    
\end{remark}
Given $f\in \cH$, we say that $f$ is \emph{cyclic for $\Mult(\cH)$} if the closed $\Mult(\cH)$-invariant subspace generated by $f$ is $\cH$, i.e. $\ol{\Mult(\cH)f}=\cH$. 
\begin{lemma}\label{lemm : density}
    Assume that $\Mult(\cH)$ is contained in $\cH$ and forms a dense subset. Let $T:\cH\to\cH$ be a bounded operator that commutes with every multiplication operator; then $T(\textbf{1})\in\Mult(\cH)$ and $T$ is the corresponding multiplication operator.
\end{lemma}
\begin{proof}
    Set $\lambda=T(\textbf{1})$. Fix $f\in\cH$, there are $\phi_n\in\Mult(\cH)$ such that $\phi_n\to f$ in $\cH$. Then
    $$T(f)=\lim _nT(\phi_n)=\lim_n TM_{\phi_n}(\textbf{1})=\lim_n M_{\phi_n}(\lambda)=\lim_n \phi_n\lambda.$$
    Recall that weak convergence in Hilbert function spaces implies pointwise convergence, so $T(f)=\lambda f$, and we obtain the result. 
\end{proof}

\begin{proposition}\label{prop : algiso2}
    Assume that $\cH_1$ and $\cH_2$ are $2$-interpolating and $\Mult(\cH_2)$ is contained in $\cH_2$ and forms a dense subset.
    If $W_{\lambda,\varphi}:\cH_2\to\cH_1$ is invertible, then the induced map $C_\varphi:\Mult(\cH_2)\to\Mult(\cH_1)$ is an isomorphism.
    Moreover, $\lambda$ is cyclic for $\Mult(\cH_1)$.
\end{proposition}
\begin{proof}
    By the first part of the proof of Proposition \ref{prop : isoalg}, the map is well-defined and injective.
    For a nonzero multiplier $\phi\in \Mult(\cH_1)$ we define a map $T:\cH_2\to \cH_2$ by $T=W_{\lambda,\varphi}^{-1}M_\phi W_{\lambda,\varphi}$. 
    We aim to show that $T$ is the multiplication operator by $\Phi=T(\textbf{1})$.
    For $\Psi\in\Mult(\cH_2)$, notice that $TM_\Psi =M_\Psi T$, indeed we have
    \begin{align*}
        W_{\lambda,\varphi}TM_\Psi W_{\lambda,\varphi}^{-1}
        &= W_{\lambda,\varphi}TW_{\lambda,\varphi}^{-1} W_{\lambda,\varphi}M_\Psi W_{\lambda,\varphi}^{-1}
        = M_\phi M_{\Phi\circ\varphi} \\
        &   = M_{\Phi\circ\varphi} M_\phi
        =W_{\lambda,\varphi}M_\Psi W_{\lambda,\varphi}^{-1} W_{\lambda,\varphi} TW_{\lambda,\varphi}^{-1} \\
        & = W_{\lambda,\varphi}M_\Psi TW_{\lambda,\varphi}^{-1}.
    \end{align*}
    One can also take $f\in\cH_2$ and show that $TM_\Psi(f)=M_\Psi T(f)$ on $\varphi(X_1)$ and therefore on $X_2$.
    Thus, by Lemma \ref{lemm : density} $\Phi\in\Mult(\cH_2)$ and $C_\varphi(\Phi)=\phi$.
    
    For the second part, fix $f\in\cH_1$, and let $F\in\cH_2$ be such that $W_{\lambda,\varphi}(F)=f$. 
    By the density of the multipliers of $\cH_2$, there is a sequence $\Phi_n\in\Mult(\cH_2)$ which converges to $F$. Thus,
    $$f=W_{\lambda,\varphi}(F)=\lim_n W_{\lambda,\varphi}(\Phi_n)=\lim_n \lambda\Phi_n\circ{\varphi}.$$
    Since $\Phi_n\circ\varphi\in\Mult(\cH_1)$ we obtain that $f\in \ol{\Mult(\cH_1)\lambda}$ and the proof is complete.
\end{proof}

\section{Weighted Hardy Spaces}
\begin{definition}
    \emph{A weighted Hardy space in $d\in\N$ variables and radius $r\in(0,\infty]$} is a Hilbert function space of holomorphic functions on $r\bB_d$, in which the monomials $z^\alpha$, $\av{\alpha}\ge 0$ form a complete orthogonal set of nonzero vectors with $\frac{\norm{z^\alpha}^2}{\alpha !}=\frac{\norm{z^\beta}^2}{\beta !}$ whenever $\av{\alpha}=\av{\beta}$ and $\norm{\textbf{1}}=1$.
\end{definition}
For $d=1$ and $r\in[1,\infty]$, we obtain the known weighted Hardy space on the unit disk $\D$ \cite{CowenCarlC, Gilad}.
 
Let $G(t)=\sum_{n=0}^\infty a_nt^n$ be a power series with $a_0=1$ and $a_n> 0$ for all $n$ and radius of convergence $R\in(0,\infty]$, we call such a function \emph{a kernel generating function} or just \emph{generating function}. 
The assumptions that $a_0=1$ and $\norm{\textbf{1}}=1$ are used to show that weighted Hardy spaces are normalized at the origin. 
The following proposition has many adaptations in the literature; for example, see \cite[Lemma 2.9]{CowenCarlC}, \cite[Proposition 4.1]{GUO2004214}, \cite[Lemma 2.2]{hartz2017isomorphism} and \cite[Proposition 1]{martin2019co}. It reveals the connection between weighted Hardy spaces and generating functions:
\begin{proposition}
    Let $\cH$ be a Hilbert function space defined on some open or closed ball in $\C^d$ with radius $r$ and kernel $K$.
    The following are equivalent:
    \begin{enumerate}
        \item $\cH$ is a weighted Hardy space in $d$ variables and radius $r$.
        
        \item There exist a generating function $G(t)=\sum_{n=0}^\infty a_nt^n$ with radius of convergence $R\ge r^2$ and $K(z,w)=G(\inp{z,w}_d)$.
    \end{enumerate}
\end{proposition}
\begin{proof}
    It is well-known that if $\inp{\cdot,\cdot}_d$ is the standard inner product on $\C^d$ then 
    $$\inp{z,w}_d^n=\sum_{\av{\alpha}=n}\frac{\av{\alpha}!}{\alpha!}z^\alpha\ol{w}^\alpha.$$

    $(1)\Rightarrow (2):$ Since the monomials $z^\alpha$, $\av{\alpha}\ge 0$ form a complete orthogonal set of nonzero vectors, by Parseval's identity, we have
    $$K(z,w)=\sum_\alpha \frac{z^\alpha \ol{w}^\alpha}{\norm{z^\alpha}^2}$$
    for all $z,w\in r\bB_d$.
    Using that $\frac{\norm{z^\alpha}^2}{\alpha !}=\frac{\norm{z^\beta}^2}{\beta !}$ whenever $\av{\alpha}=\av{\beta}$, we define a sequence $a_n$ such that $\norm{z^\alpha}^{-2}=a_{\av{\alpha}}\frac{\av{\alpha}!}{\alpha!}$.
    Hence, 
    $$K(z,w)
    =\sum_\alpha a_{\av{\alpha}}\frac{\av{\alpha}!}{\alpha!} z^\alpha\ol{w}^\alpha
    =\sum_{n=0}^\infty a_n\sum_{\av{\alpha}=n}\frac{\av{\alpha}!}{\alpha!}z^\alpha\ol{w}^\alpha
    =\sum_{n=0}^\infty a_n\inp{z,w}_d^n.$$
    So, we define $G(t)=\sum_{n=0}^\infty a_nt^n$. 
    Notice that $a_0=\frac{1}{\norm{\textbf{1}}^2}=1$ and $a_n>0$ for all $n$.
    It remains to show that the radius of convergence $R$ is at least $r^2$.
    Assume for the sake of reaching a contradiction that $R<r^2$, let $t_0\in(R,r^2)$, then for a point $x\in r\bB_d$ with norm $t_0$, we obtain $G(t_0)=K(x,x)<\infty$, which is a contradiction.

    $(2)\Rightarrow (1):$ This is \cite[Proposition 4.1]{GUO2004214}.
    We add the proof for completeness.
    We set $e_\alpha=\pare{a_{\av{\alpha}}\frac{\av{\alpha}!}{\alpha!}}^{\nicefrac{1}{2}} z^\alpha$, and let $H$ be the Hilbert space such that $\set{e_\alpha:\alpha\text{ is a multi-index}  }$ is its orthonormal basis. 
    Our goal is to show that the spaces $H$ and $\cH$ coincide.
    For $h\in H$ we can write $h=\sum_\alpha b_\alpha e_\alpha$ with $\sum_\alpha \av{b_\alpha}^2<\infty$. 
    For $w\in r\bB_d$ let $\on{ev}_w(h)$ be the evaluation of $h$ at $w$ which satisfies:
    \begin{align*}
        \av{\on{ev}_w(h)}=&\av{\sum_\alpha b_\alpha e_\alpha(w)}
        \le \pare{\sum_\alpha \av{b_\alpha}^2}^{\nicefrac{1}{2}}\pare{\sum_\alpha \av{e_\alpha(w)}^2}^{\nicefrac{1}{2}} \\
        =&\pare{\sum_\alpha \av{b_\alpha}^2}^{\nicefrac{1}{2}}\sqrt{G(\av{w}^2)}=\sqrt{G(\av{w}^2)}\norm{h}.
    \end{align*}
    So, we can think of $H$ as a Hilbert function space on $r\bB_d$.
    Moreover, since $G$ is an increasing function of $\av{w}$, this bound shows that convergence in $H$ yields uniform convergence on compact subsets of $r\bB_d$.
    Since the orthonormal basis of $H$ consists of polynomials, each function in $H$ is a uniform limit of polynomials on compact subsets of $r\bB_d$ and hence is a holomorphic function on $r\bB_d$.
    Let $K^H$ be the kernel of $H$ then   
    \begin{gather*}
        K^H(z,w)=\sum_\alpha \ol{e_\alpha(w)}e_\alpha(z)
        =\sum_\alpha a_{\av{\alpha}}\frac{\av{\alpha}!}{\alpha!} z^\alpha\ol{w}^\alpha
        =G(\inp{z,w}_d)=K(z,w).
    \end{gather*}
    By the uniqueness part of Moore's theorem \cite{moore1939generalanalysis2}, we must have $H = \cH$.
    Finally, we see that $\norm{z^\alpha}^{-2}=a_{\av{\alpha}}\frac{\av{\alpha}!}{\alpha!}$, so the monomials are nonzero vectors with $\frac{\norm{z^\alpha}^2}{\alpha!}=\frac{\norm{z^\beta}^2}{\beta !}$ whenever $\av{\alpha}=\av{\beta}$ and $\norm{\textbf{1}}=1$.
    This shows that $\cH$ is a weighted Hardy space in $d$ variables and radius $r$, completing the proof. 
\end{proof}
\begin{remark}
    In fact, every generating function gives rise to a weighted Hardy space; for example, see \cite[Theorem 4.14]{Paulson}.
    Thus, for fixed $d$, there is a bijective correspondence between weighted Hardy spaces in $d$ variables and generating functions.
\end{remark}
\begin{corollary}
    Let $\cH$ be a weighted Hardy space in $d$ variables with generating function $G$; then $\cH$ is normalized at the origin and fully interpolating.
\end{corollary}
\begin{proof}
    First, for any point $x$ we have $k_x=G(\inp{\cdot,x}_d)$, so  $k_x(0)=G(\inp{0,x}_d)=G(0)=1$ which shows that $\cH$ is normalized at the origin and $k_x$ is not the zero function.
      
    Second, since $\cH$ contains the polynomials, for every natural $m$, it is easy to find $m$ polynomials interpolating $m$ points.
    Hence, by definition, $\cH$ is fully interpolating.
\end{proof}

The following lemma is inspired by \cite[Theorem 2.15]{CowenCarlC} and \cite[Lemma 5.3]{hartz2017isomorphism}. 
Roughly speaking, it shows that for a weighted Hardy space, the ball is a maximal domain.
\begin{lemma}\label{lem : domof}
    Let $\cH$ be a weighted Hardy space in $d$ variables with generating function $G(t)=\sum_{n=0}^\infty a_nt^n$. Let $r^2\in(0,\infty]$ be the radius of  convergence of $G$ then,
    \begin{enumerate}
    \item If $r=\infty$, then $\C^d$ is a maximal domain for $\cH$.

    \item If $r<\infty$ and $G(r^2) = \infty$, then $r\bB_d$ is a maximal domain for $\cH$.
    
    \item If $r<\infty$ and $G(r^2) < \infty$, then $\ol{r\mathbb{B}_d}$ is a maximal domain for $\cH$.
    \end{enumerate}
\end{lemma}
\begin{proof}
    Assume that $\rho$ is a multiplicative functional on $\cH$, let $w=(w_1,\dots,w_d)$ where $w_i=\rho(z_i)$. 
    It follows that 
    $$\rho(z^\alpha)=\rho(z_1^{\alpha_1})\cdots\rho(z_d^{\alpha_d})=\rho(z_1)^{\alpha_1}\cdots\rho(z_d)^{\alpha_d}=w^\alpha$$
    for any multi-index with $\av{\alpha}\ge 1$. 
    Recall that $\rho(\textbf{1})=1$. 
    Let us show that $w \in \ol{r\bB_d}$ (when $r=\infty$, there is nothing to prove). 
    To this end, let $\zeta\in r\bB_d$ then since $\cH$ contains the the monomials we have  $\inp{\cdot,\zeta}_d^n\in\cH$ for all natural $n$, thus
    $$\rho(\inp{\cdot,\zeta}_d)=\sum_{i=1}^d w_i\overline{\zeta_i}=\inp{w,\zeta}_d$$
    and
    $$\rho(k_\zeta)=\sum_{n=0}^\infty a_n\rho(\inp{\cdot,\zeta}_d^n)=\sum_{n=0}^\infty a_n\rho(\inp{\cdot,\zeta})_d^n=\sum_{n=0}^\infty a_n\inp{w,\zeta}_d^n$$
    Hence, $\av{\inp{w,\zeta}_d}\le r^2$, and since $\zeta\in r\bB_d$ was arbitrary, we conclude that $w\in\ol{r\bB_d}$.
    
    Let $k\in \cH$ be such that $\rho(f)=\inp{f,k}$ for all $f\in\cH$. As the monomials are orthogonal and span a dense subspace of $\cH$, we see that
    $$k(z)=\inp{k,\textbf{1}}1+\sum_{\alpha}\ol{\rho(z^\alpha)}\frac{z^\alpha}{\norm{z^\alpha}^2}
    =\sum_\alpha\frac{\ol{w}^\alpha}{\norm{z^\alpha}^2}z^\alpha$$
    for all $z\in r\bB_d$ by pointwise convergence. 
    If $w\in r\bB_d $, we see that $k=G(\inp{\cdot,w}_d)=k_w$ which is the linear functional of evaluation at $w$, in particular this proves (1).
    If $\av{w}= r$ then:
    $$\norm{k}^2=\sum_\alpha\frac{\ol{w}^\alpha w^\alpha}{\norm{z^\alpha}^2}=G(\inp{w,w}_d)= G(r^2).$$
    Hence, if $G(r^2)=\infty$, this contradicts the boundedness of the linear functional determined by $k$, and we obtain (2).
    For (3), we observe that the kernel of $\cH$ can be extended to a kernel on $\ol{r\bB_d}\times\ol{r\bB_d}$, hence all functions in $\cH$ extend to functions on $\ol{r\bB_d}$, and $\cH$ becomes a Hilbert function space on $\ol{r\bB_d}$ in this way and $\overline{r\mathbb{B}_d}$ is a maximal domain for $\cH$.
\end{proof}
From now on, when we say that $\cH$ is a weighted Hardy space in $d$ variables and radius $r\in(0,\infty]$, we consider the space defined on the suitable domain aligned with the last lemma with the same $r$. This domain is frequently called \emph{the natural maximal domain}.

\medskip
Let $\cH$ be a weighted Hardy space with the natural maximal domain $X$. 
If the generating function of $\cH$ is entire, then $\Mult(\cH)$ is just the constant functions because $\Mult(\cH)\subseteq H^\infty(X)$ as sets. 
Accordingly, $\cH$ is not algebraically consistent on $X$ as every function in $\cH$ that gets one at the origin induces a partially multiplicative functional on $\cH$. 
We aim to classify all the algebraically consistent weighted Hardy spaces, starting with the following lemma.
\begin{lemma}\label{lem : cordinat}
    Let $\cH$ be a weighted Hardy space in $d$ variables. 
    Suppose that $z_i\notin\Mult(\cH)$ for some $1\leq i\leq d$.
    For $\lambda\in \Mult(\cH)$ write $\lambda(z)=\sum_\alpha \lambda_\alpha z^\alpha$, then $\lambda_{e_i}=0$ where
    $e_i$ is the $i$-th standard multi-index.
\end{lemma}
\begin{proof}
    For every multi-index $\beta$, orthogonality of the monomials gives
    $$\norm{\lambda z^\beta}^2=\sum_\alpha \av{\lambda_\alpha}^2\norm{z^{\alpha+\beta}}^2\ge\av{\lambda_{e_i}}^2\norm{z^{e_i+\beta}}^2.$$
    On the other hand,
    $$\norm{\lambda z^\beta}^2\le \norm{M_\lambda}^2\norm{z^\beta}^2.$$
    Thus, if $\lambda_{e_i}\neq0$, then
    $$\norm{z^{e_i+\beta}}\le \frac{\norm{M_\lambda}}{\av{\lambda_{e_i}}}\norm{z^\beta}$$
    for every multi-index $\beta$.
    Consequently, for every polynomial $p(z)=\sum_\beta p_\beta z^\beta$,
    $$\norm{z_ip}^2=\sum_\beta \av{p_\beta}^2\norm{z^{{e_i}+\beta}}^2\le\frac{\norm{M_\lambda}^2}{\av{\lambda_{e_i}}^2}\sum_\beta \av{p_\beta}^2\norm{z^{\beta}}^2=\frac{\norm{M_\lambda}^2}{\av{\lambda_{e_i}}^2}\norm{p}^2.$$
    Thus, multiplication by $z_i$ extends boundedly to $\cH$, contrary to the assumption that $z_i\notin\Mult(\cH)$. 
    Hence $\lambda_{e_i}=0$.
\end{proof}
We now obtain a practical characterization of algebraically consistent weighted Hardy spaces.
\begin{theorem}\label{thm : graet}
    Let $\cH$ be a weighted Hardy space in $d$ variables. 
    The following are equivalent:
    \begin{enumerate}
        \item $\Mult(\cH)$ contains the coordinate functions.
        \item $\Mult(\cH)$ forms a dense subset of $\cH$.
        \item There exists $c\in\cH$ such that $\ol{\Mult(\cH)c}=\cH$.
        \item $\cH$ is algebraically consistent on its natural maximal domain.
    \end{enumerate}
\end{theorem}
\begin{proof}
    Plainly, $(1)\Rightarrow(2)\Rightarrow(3)$. 
    
    For $(1)\Rightarrow(4)$, we need to show that if $\rho$ is a \emph{partially} multiplicative functional on $\cH$, then it comes from point evaluation. 
    Set $w=(\rho(z_1),\dots,\rho(z_d))$, then since $\Mult(\cH)$ contains all the monomials we have $\rho(z^\alpha)=w^\alpha$ for every multi-index $\alpha$.
    Now, repeat the proof as in Lemma \ref{lem : domof} to obtain that $\rho$ is the point evaluation at $w$. The implication can also be deduced from \cite[Theorem 1.7]{MR4225497}. 

    Next we prove that $(3)\Rightarrow(1)$. 
    Assume that $z_i\notin\Mult(\cH)$ for some $1\leq i\leq d$.
    Now suppose that $c(z)=\sum_\alpha c_\alpha z^\alpha\in\cH$. We show that $c$ cannot be cyclic for $\Mult(\cH)$. 
    Let $\lambda\in\Mult(\cH)$, write $\lambda(z)=\sum_\alpha \lambda_\alpha z^\alpha$, then
    $(\lambda c)_0=\lambda_0c_0$ and $(\lambda c)_{e_i}=\lambda_0c_{e_i}+\lambda_{e_i}c_0=\lambda_0c_{e_i}$ by Lemma \ref{lem : cordinat}. 
    If $c_0=0$,  then $\inp{\lambda c,k_0}=0$ for every multiplier $\lambda\in\Mult(\cH)$.
    Otherwise, set $p(z)=\ol{c_{e_i}}-\frac{\ol{c_0}}{\norm{z_i}^2}z_i$. For $f\in\cH$ write $f(z)=\sum_\alpha f_\alpha z^\alpha$ and then $\inp{f,p}=c_{e_i}f_0-c_0f_{e_i}$.
    Thus,
    $$\inp{\lambda c,p}=c_{e_i}\lambda_0 c_0-c_0 \lambda_0 c_{e_i}=0$$
    for every multiplier $\lambda\in\Mult(\cH)$. 
    In either case, $\Mult(\cH)c$ is contained in the kernel of a nonzero bounded linear functional and therefore cannot be dense in $\cH$.
    
    Finally, we show $(4)\Rightarrow(1)$: Assume that $z_i\notin\Mult(\cH)$ for some $1\leq i\leq d$. Set $p(z)=1+z_i$ and define the linear functional $\rho(f)=\inp{f,p}$. We show that $\rho$ is a partially multiplicative functional.
    For $\lambda\in\Mult(\cH)$ and $f\in\cH$ write $\lambda(z)=\sum_\alpha \lambda_\alpha z^\alpha$ and $f(z)=\sum_\alpha f_\alpha z^\alpha$. Thus, $\rho(f) =f_0+f_{e_i}\norm{z_i}^2$ and by Lemma \ref{lem : cordinat} $\rho(\lambda)=\lambda_0+\lambda_{e_i}\norm{z_i}^2=\lambda_0$.
    Hence
    \begin{align*}
        \rho(\lambda f) & = \lambda_0 f_0+ (\lambda_{e_i}f_0+\lambda_0 f_{e_i})\norm{z_i}^2 \\
        & = \lambda_0(f_0+f_{e_i}\norm{z_i}^2)=\rho(\lambda)\rho(f). 
    \end{align*}
    Clearly, $\rho$ is not a point evaluation in the natural maximal domain, and therefore $\cH$ is not algebraically consistent.
\end{proof}
When $\cH$ is a weighted Hardy space with the complete Pick property, $\Mult(\cH)$ contains the monomials; see, for example, \cite[Section 4]{greene2002}. Then, by Theorem \ref{thm : graet}, $\cH$ is algebraically consistent on its natural maximal domain; 
this is a version of \cite[Corollary 3.4]{MR3687947} for weighted Hardy space with the complete Pick property. 
In fact, we now obtain a stronger result, assuming only the scalar Pick property.
\begin{corollary}
    Let $\cH$ be a weighted Hardy space with the scalar Pick property and let $\rho$ be a bounded non-zero functional on $\cH$. Then the following are equivalent.
    \begin{enumerate}
        \item $\rho(\phi g)=\rho(\phi)\rho(g)$ for all $\phi\in\Mult(\cH)$ and all $g\in\cH$;
        \item $\rho(fg)=\rho(f)\rho(g)$ whenever $f,g\in\cH$ such that $fg\in\cH$.
    \end{enumerate}
    
\end{corollary}
\begin{proof}
    $(2)\Rightarrow (1)$ is immediate.
    For $(1)\Rightarrow (2)$, it is equivalent to prove that a weighted Hardy space in $d$ variables with the scalar Pick property is algebraically consistent on its natural maximal domain.
    
    Let $G$ be the generating function of $\cH$. There is $R\in(0,\infty)$ such that $G$ is holomorphic and never vanishes on $R^2\D$. 
    Set $S=R\bB_d$ and denote by $\cH|_{S}$ the restricted space.
    Since $\cH$ has the scalar Pick property, it is not hard to see that the restriction map $\Mult(\cH)\to\Mult(\cH|_{S})$, $\lambda\mapsto\lambda|_{S}$ is a bijection and $\cH|_{S}$ also has the scalar Pick property. 
    For $w\in S$, $\psi_w=1-\frac{1}{k_w|_{S}}$ is a strictly contractive multiplier of $\cH|_{S}$ by \cite[Proposition 3.1]{hartz2017isomorphism}. Thus, $k_w|_{S}=\sum_n\psi_w^n$ is also a multiplier of $\cH|_{S}$.
    By the identity theorem, we obtain that $k_w\in \Mult(\cH)$ for all $w\in S$ and the set $\set{k_w:w\in S}$ spans a dense subset of $\cH$, in particular, $\Mult(\cH)$ forms a dense subset. 
    Therefore, Theorem \ref{thm : graet} yields that $\cH$ is algebraically consistent on its natural maximal domain, and the proof is complete. 
\end{proof}

We now conclude another well-known property of the weighted Hardy spaces. 
A reproducing kernel $K$ on a ball $X\subseteq\C^d$ is said to be \emph{unitary invariant} if it satisfies $K(Uz, Uw)=K(z,w)$ for all $z,w\in X$ and all unitary maps $U\in\mathcal{U}(d)$.
\begin{corollary}\label{lem : unitary}
    Let $\cH$ be a weighted Hardy space in $d$ variables with generating function $G$.
    Then the kernel of $\cH$ is unitary invariant and $C_U:\cH\to\cH$ is a unitary for all $U\in\mathcal{U}(d)$.
\end{corollary}
For the proof, we need a simple lemma:
\begin{lemma}\label{lem : kernel relation}
    Let $\cH_i$ be a Hilbert function space with maximal domain $X_i$ and kernel function $K_i$, $i=1,2$.
    Suppose that $\varphi:X_1\to X_2$ is a bijection and $K_1(x,y)=K_2(\varphi(x),\varphi(y))$ for any $x,y\in X_1$.
    Then the map $C_\varphi:\cH_2\to\cH_1$ is unitary.
\end{lemma}
\begin{proof}
    By Theorem \ref{Thm : bound}, $C_\varphi$ and $C_{\varphi^{-1}}$ are bounded.
    Moreover, $C_\varphi$ is the adjoint of the map $k^1_x \mapsto k^2_{\varphi(x)}$, which, by assumption, is an isometry.
    Since $ C_{\varphi^{-1}}C_\varphi=I_{\cH_2}$, $C_\varphi^\ast$ is an onto isometry, i.e., a unitary map.
\end{proof}
\begin{proof}[\textbf{Proof of Corollary \ref{lem : unitary}}]
    Fix $U\in\mathcal{U}(d)$, so $G(\inp{Uz,Uw}_d)=G(\inp{z,w}_d)$ and $\cH$ is unitary invariant.
    The second part follows immediately from Lemma \ref{lem : kernel relation}.
\end{proof}

\begin{remark}
    From the corollary above and Proposition \ref{prop : isoalg}, we see that for a weighted Hardy space $\cH$ in $d$ variables, each $U\in\mathcal{U}(d)$ induces the isometric isomorphism $\Mult(\cH)\to\Mult(\cH)$ given by $\lambda\mapsto \lambda\circ U$. 
    Therefore, if $z_i\in\Mult(\cH)$ for some $1\le i\le d$, then $z_j\in\Mult(\cH)$ for all $1\le j\le d$, since changing coordinates in $\C^d$ is a unitary map.
\end{remark}

We close the section with some examples for weighted Hardy spaces:
\begin{example}\label{ex : space}
    \begin{enumerate}
        \item For a real number $s$, let $\cH_s$ be the weighted Hardy space in one variable with the generating function
        $$G_s(t)=\sum_{n=0}^\infty (n+1)^st^n.$$
        It is easy to see that $\D$ is the natural maximal domain for $\cH_s$ with $s\ge -1$, and $\ol{\D}$ is the natural maximal domain for $\cH_s$ with $s<-1$. Since $\sup_n \frac{n^s}{(n+1)^s}<\infty$, $z\in\Mult(\cH_s)$ for all $s$. 
        Then, by Theorem \ref{thm : graet}, $\cH_s$ is algebraically consistent on its natural maximal domain.
        
        Where $s=1$, $\cH_s$ is the \emph{Bergman space $A^2(\D)$}. 
        Where $s=0$, $\cH_s$ is the \emph{Hardy space $H^2(\D)$}.
        It is well known that the multiplier algebras of the Bergman and Hardy spaces are isometrically isomorphic to $H^\infty(\D)$; see, for example, \cite[Theorem 4.6]{Jim}.
        
        Where $s=-1$, $\cH_s$ is the \emph{Dirichlet space $\cD$}. 
        By \cite[Exercise 2.l.l2]{CowenCarlC}, its multiplier algebra is a proper subset of $H^\infty(\D)$.

        For $s\le 0$, the space $\cH_s$ has the complete Pick property \cite[Corollary 7.41]{Jim}, while the Bergman space $A^2(\D)$ does not have the Pick property \cite[Example 5.1]{Jim}.
        
        \item The \emph{Segal--Bargmann space $\cS$} is the weighted Hardy space in one variable with the generating function
        $$G(t)=\sum_{n=0}^\infty\frac{1}{n!}t^n=e^t.$$
        Since $G$ is entire, we observe that $\C$ is the natural maximal domain for $\cS$ and $\Mult(\cS)$ is just the constant functions.
      
        \item The \emph{Drury-Arverson space $H^2_d$} is the weighted Hardy space in $d$ variables with the generating function $G_0(t)=\frac{1}{1-t}$.
        For $d=1$, we obtain the Hardy space. For $d\ge 2$, the space of multipliers is a proper subset of $H^\infty(\bB_d)$; see Remark 8.9 and the paragraph following it in \cite{Jim}.  
        Nevertheless, those spaces have the complete Pick property and are hence algebraically consistent on their natural maximal domain $\bB_d$.
    \end{enumerate}
\end{example}

\section {Weighted isomorphisms between weighted Hardy spaces}
In view of Proposition \ref{prop : isoalg}, if $\Mult(\cH_1)$ and $\Mult(\cH_2)$ are not isomorphic via a composition operator, then $\cH_1$ and $\cH_2$ are not isomorphic via an RKHS isomorphism. 
However, studying the multiplier algebras is not enough to determine if two spaces are isomorphic via an RKHS isomorphism. 
A well-known example is the relation between the Hardy and Bergman spaces. Both have $H^\infty(\D)$ as their multiplier algebra and $C_{\on{id}}:\Mult(A^2(\D))\to\Mult(H^2(\D))$ is an isomorphism. 
However, we claim that:  
\begin{claim}
    $\cH_s$ is not weakly related to $A^2(\D)$ for all $s\le 0$.
\end{claim}
\begin{proof}
    We can represent the inner product of $A^2(\D)$ by
    $$\inp{f,g}_{A^2(\D)}=\frac{1}{\pi}\int_\D f(z)\ol{g(z)}dA(z)$$
    where $dA(z)$ is the Lebesgue area measure on the unit disk; see \cite[Example 2.2]{Jim}. 
    With this representation, it is easy to obtain that $C_\psi:A^2(\D)\to A^2(\D)$ is a well-defined isomorphism for all $\psi\in\on{Aut}(\D)$. 
    
    Fix $s\le 0$ and assume, for the sake of contradiction, that $\cH_s$ is weakly related to $A^2(\D)$ and
    $W_{\lambda,\varphi}: A^2(\D)\to \cH_s$ is invertible.
    By Proposition \ref{prop : dim}, we must have $-1\le s\le 0$ (because $\ol{\D}$ is the natural maximal domain for $\cH_s$ whenever $s<-1$) and $\varphi\in\on{Aut}(\D)$.
    Thus, $M_\lambda=W_{\lambda,\varphi}C_{\varphi^{-1}}: A^2(\D)\to\cH_s$ is invertible.
    Hence, for $c=\norm{ M_\lambda^\ast}^2$ we have 
    $$\av{\lambda(z)}^2G_1(\inp{z,z}_1)\le cG_s(\inp{z,z}_1)\le cG_0(\inp{z,z}_1)$$
    for all $z\in \D$. That is,
    \begin{gather*}
        \frac{\av{\lambda(z)}^2}{(1-\av{z^2})^2}\le \frac{c}{1-\av{z}^2} 
    \end{gather*}
    which is the same as 
    $$\frac{1}{1-\av{z}^2}\le c\frac{1}{\av{\lambda(z)}^2}$$
    for all $z\in \D$, where we have used that $\lambda$ is never vanishes (Lemma \ref{lem : inj}). 
    This leads us to a contradiction, since $\nicefrac{1}{\lambda}\in A^2(\D)$ and integrating the last inequality gives
    $$\int_\D \frac{dA(z)}{1-\av{z}^2} \le c \int_\D\frac{dA(z)}{\av{\lambda(z)}^2}.$$
    As the left-hand side diverges, we obtain the contradiction.
\end{proof}

This leads us to some natural questions: 
\begin{quest}
Let $\cH_i$ be a weighted Hardy space with the natural maximal domain $X_i\subseteq \C^{d_i}$, $i=1,2$.
When are $\cH_1$ and $\cH_2$ isomorphic via an RKHS isomorphism? 
When are they strongly related? And when is $\cH_1$ weakly related to $\cH_2$?
\end{quest}
We present an example that illustrates one of the difficulties of working with composition operators. The following example shows that we cannot always split a weighted composition operator into a ``natural" product of a composition operator and a multiplication operator.
\begin{example}
    We build an example of a composition operator $W_{\lambda,\varphi}$ which does not arise as the composition $M_\lambda C_\varphi$, where $\lambda\in\Mult(\cH_2,\cH_1)$ and $C_\varphi:\cH_2\to\cH_2$ or $\lambda\in\Mult(\cH_1)$ and $C_\varphi:\cH_2\to\cH_1$.
    
    To this end, we use an example from \cite{lefevre2025characterization}. Let $G(t)=\sum_{n=0}^\infty a_nt^n$ where $a_0=a_1=1$ and $a_n=k!$ when $k!<n\le (k+1)!$. Denote by $\cH$ the weighted Hardy space in one variable with $G$ as its generating function. 
    Notice that $\D$ is the natural maximal domain for $\cH$.
    Let $a\in \D\sm\set{0}$ and $\varphi_a\in\on{Aut}(\D)$ be the involution map.
    \cite[Theorem 1.1]{lefevre2025characterization} yields that $C_{\varphi_a}$ is not bounded operator on $\cH$.
    Let $\cK$ be the Hilbert function space on $\D$ with the kernel $K(x,y)=xG(\inp{\varphi_a(x),\varphi_a(y)}_1)\ol{y}$, so $\cK=\set{zf(\varphi_a(z)):f\in\cH}$, see \cite[Theorem 3.1]{kumari2025composition}.
    Notice that by Theorem \ref{Thm : bound}, $W_{z,\varphi_a}:\cH\to\cK$ is a bounded weighted composition operator.
    Assume that $W_{z,\varphi_a}=M_zC_{\varphi_a}$ where $M_z$ acts on either $\cH$ or $\cK$.
    Since $C_{\varphi_a}$ is not bounded operator on $\cH$, we must have that $C_{\varphi_a}:\cH\to\cK$.
    However, that means that $C_{\varphi_a}(\textbf{1})=\textbf{1}\in\cK$.
    If $\textbf{1}=zf(\varphi_a(z))$ for $f\in\cH$, then $\varphi_a(z)f(z)=1$ for all $z\in\D$. 
    This contradicts the fact that $f$ is holomorphic on $\D$. 
\end{example}

We continue our journey with the following easy conclusions:
\begin{proposition}\label{cor : sr}
     Let $\cH_i$ be a weighted Hardy space in $d_i$ variables with generating function  $G_i$ with radius of convergence $r_i^2\in(0,\infty]$, $i=1,2$. 
     Assume that $\cH_1$ is weakly related to $\cH_2$ via $W_{\lambda,\varphi}:\cH_2\to\cH_1$, then: 
    \begin{enumerate}
        \item If $d_1\ge d_2$, then $\cH_1$ and $\cH_2$ are isomorphic via an RKHS isomorphism and $d_1=d_2$.
        \item If $\cH_1$ and $\cH_2$ strongly related, then $\cH_1$ and $\cH_2$ are isomorphic via an RKHS isomorphism and $d_1=d_2$.
        \item If $\cH_1$ and $\cH_2$ strongly related, then $G_1$ is entire if and only if $G_2$ is entire.
        \item If $\cH_1$ and $\cH_2$ strongly related and $r_1<\infty$, then $r_2<\infty$ and $G_1(r_1^2)<\infty$ if and only if $G_2(r_2^2)<\infty$.
    \end{enumerate}
\end{proposition}
\begin{proof}
    If $d_1\ge d_2$, by Proposition \ref{prop : dim} we obtain that $d_1=d_2$ and $W_{\lambda,\varphi}^{-1}$ is a weighted composition operator, which shows that $\cH_1$ and $\cH_2$ are isomorphic via an RKHS isomorphism. 
    
    (2) is immediate from (1). For the rest of the proof, we denote by $X_i$ the natural maximal domain for $\cH_i$, $i=1,2$. 
    
    For the third part assume that $G_1$ is entire, then by Proposition \ref{prop : dim} and Lemma \ref{lem : domof}, there is a bijection $\varphi:\C^{d_1}\to X_2\subseteq \C^{d_1}$ which is holomorphic on $\C^{d_1}$. 
    Since $X_2\in\set{r_2\bB_{d_1},\ol{r_2\bB_{d_1}}}$ where $r_2\in (0,\infty]$, by Liouville’s theorem, we must have $X_2=\C^{d_1}$ and $r_2=\infty$, which means that $G_2$ is entire too.

    For the last part, if $G_1(r_1^2)<\infty$ there is a homeomorphism $\varphi:\ol{r_1\bB_{d_1}}\to X_2$. Since  $\ol{r_1\bB_{d_1}}$ is compact we must have that $X_2=\ol{r_2\bB_{d_1}}$ with $r_2<\infty$ and then $G_2(r_2^2)<\infty$.
\end{proof}

Currently, it is not known whether a strong relation always implies the existence of an RKHS isomorphism between general $2$-interpolating Hilbert function spaces. 
Notice that Example \ref{Ex : 1} shows that a weak relation does not imply a strong relation in the setting of one of the spaces being a unitary invariant space. 
We conclude this section by showing that under a reasonable additional assumption, the existence of an invertible weighted composition operator implies the existence of an RKHS isomorphism.
\begin{proposition}\label{prop : new}
     Let $\cH_i$ be a weighted Hardy space with the natural maximal domain $X_i\subseteq \C^{d_i}$, $i=1,2$.
     Assume that $\cH_2$ is algebraically consistent on $X_2$.
     If $W_{\lambda,\varphi}:\cH_2\to\cH_1$ is invertible, then $\varphi:X_1\to X_2$ is a bijection and $W_{\lambda,\varphi}^{-1}$ is a weighted composition operator.
\end{proposition}
    
\begin{proof}
    By Proposition \ref{prop : dim}, it is sufficient to show that $d_2\le d_1$.
    Since $\varphi$ is injective and continuous on $X_1$ (Lemma \ref{lem : inj}), it is sufficient to find an open subset $B$ of $X_2$ such that $B\subset \varphi(X_1)$. 
        
    For $\phi\in\Mult(\cH_1)$ and $g\in \cH_1$, there are some $F,G,H,K\in\cH_2$ such that
    $$W_{\lambda,\varphi}(F)=\phi,W_{\lambda,\varphi}(G)=g,W_{\lambda,\varphi}(H)=\phi g,W_{\lambda,\varphi}(K)=\textbf{1}.$$   
    By Theorem \ref{thm : graet}, $\Mult(\cH_2)$ form a dense subset of $\cH_2$ and then by Proposition \ref{prop : algiso2} there exists a multiplier $\Phi\in \Mult(\cH_2)$ such that $\Phi\circ\varphi=\phi$ and $\lambda$ is cyclic for $\Mult(\cH_1)$.
    Moreover, $\cH_1$ is algebraically consistent on $X_1$ by Theorem \ref{thm : graet}.
    Notice that $W_{\lambda,\varphi}(\Phi)=\lambda \phi$ and $\Phi\circ\varphi\cdot \lambda G\circ\varphi=\lambda H\circ\varphi$,
    so $\Phi\circ\varphi\cdot  G\circ\varphi= H\circ\varphi$.
    Since every function is determined by its values on $\varphi(X_1)$, we have $H=\Phi G$ on $X_2$.
    Notice that 
    $$W_{\lambda,\varphi}(\Phi K)=\lambda\Phi\circ\varphi K\circ\varphi=\phi=W_{\lambda,\varphi}(F).$$
    By injectivity, $F=\Phi K$.
    Fix $y\in X_2\setminus K^{-1}(0)$ and set $c=K(y)$. 
    As $W_{\lambda,\varphi}^\ast$ is onto, there is some $h \in\cH_1$ such that $W_{\lambda,\varphi}^\ast(h)=k_y^2$.
    We have
    \begin{align*}
        c\inp{\phi g,h}_{\cH_1}
        & =c\inp{W_{\lambda,\varphi}(H),h}_{\cH_1}
        =c\inp{H,k_y^2}_{\cH_2} \\
        & =K(y)H(y)=K(y)\Phi(y)G(y)= F(y)G(y) \\
        & =\inp{F,k_y^2}_{\cH_2} \inp{G,k_y^2}_{\cH_2} 
        =\inp{\phi,h}_{\cH_1}\inp{ g,h}_{\cH_1}.
    \end{align*}
   
    As usual, set $k=\frac{1}{\ol{c}}h$ and obtain by the observation that $\cH_1$ is algebraically consistent that $k=k_x^1$ for some $x \in X_1$ and $\varphi(x)=y$.
    Since $K$ is continuous on $X_2$, there is an open neighborhood $B\subset X_2$ of $\varphi(0)$ so that $K$ does not vanish on $B$. Thus, $B$ lies in the image of $\varphi$, and the proof is complete.
\end{proof}
It is not known whether we can omit the assumption that $\cH_2$ is algebraically consistent without any further assumptions.

\section{Classification of weighted Hardy spaces}
In \cite{Gilad}, the authors classified the weighted Hardy spaces in one variable up to RKHS isomorphism whose adjoint is an invertible composition operator, i.e., a weighted composition operator with a constant multiplier symbol. 
Our goal is to solve this problem even when the multiplier symbol need not be a constant, and for weighted Hardy spaces on arbitrary balls. 
In this section, we prove Theorems \ref{thm : big1} and \ref{thm : big2}, along with several complementary results.
\begin{theorem}\label{thm : big1}
    Let $\cH_1$ be weighted Hardy space with the generating function $G_1(t)=\sum_{n=0}^\infty a_nt^n$ and $\cH_2$ be weighted Hardy space with the generating function $G_2(t)=\sum_{n=0}^\infty b_nt^n$. 
    Assume that $G_1$ is not entire.
    If the spaces are strongly related, then there exist $C,c,\gamma\in (0,\infty)$ such that $c\le \frac{b_n}{a_n}\gamma^n\le C$ for all $n$. 
\end{theorem}
\begin{theorem}\label{thm : big2}
   Let $\cH_1$ be a weighted Hardy space in one variable with the generating function $G_1(t)=\sum_{n=0}^\infty a_nt^n$ and $\cH_2$ be a weighted Hardy space with the generating function $G_2(t)=\sum_{n=0}^\infty b_nt^n$. 
   If the spaces are strongly related, then there exist $C,c,\gamma\in (0,\infty)$ such that $c\le \frac{b_n}{a_n}\gamma^n\le C$ for all $n$. 
\end{theorem}

The following proposition is a crucial step toward our goals.
 Its proof uses an adaptation of the well-known ``disc trick'', originally introduced by Solel and Shalit in \cite[Section 11]{Disk}, and was later used in \cite{davidson2011isomorphism, hartz2017isomorphism, Gilad} to study various isomorphism problems.
\begin{proposition}\label{lem : Orr}
     Let $\cH_1$ and $\cH_2$ be weighted Hardy spaces in $d$ variables.
     Assume that $W_{\lambda,\varphi}:\cH_2\to\cH_1$ is invertible; then there exists an invertible weighted composition operator with a composition symbol that fixes the origin.  
\end{proposition} 
\begin{proof}
    Let $r_i^2\in (0,\infty]$ be the radius of convergence of the generating function of $\cH_i$.
    If $\varphi(0)\neq 0$ set
    $$\mathcal{O}_1=\set{w\in r_1\bB_d :\psi(w)=0, W_{\delta,\psi}:\cH_1\to\cH_1\text{ is invertible}},$$
    and
    $$\mathcal{O}_2=\set{w\in r_2\bB_d :\psi(0)=w, W_{\delta,\psi}:\cH_2\to\cH_1\text{ is invertible}}.$$
    Our goal is to show that $0\in\mathcal{O}_2$. 
    Since $C_V$ is an automorphism of $\cH_2$ for every unitary $V\in\mathcal{U}(d)$, $W_{\lambda,\varphi} C_V=W_{\lambda,V\varphi}:\cH_2\to\cH_1$ is still invertible, we obtain that
    $$ S=\set{V\varphi(0):V\in\mathcal{U}(d)}\subset\mathcal{O}_2 .$$
    We see that $S=\set{z\in\C^d:\av{z}=\av{\varphi(0)}}$, which does not contain the origin.
    Notice that if $W_{\delta,\psi}:\cH_2\to\cH_1$ is invertible, then $W_{\delta,\psi}^{-1}$ is a weighted composition operator by Proposition \ref{cor : sr}.
    Thus, $W_{\lambda,\varphi}W_{\delta,\psi}^{-1}$ is a weighted composition automorphism of $\cH_1$ with composition symbol $\psi^{-1}\circ\varphi$.
    Using $\psi=V\varphi$ with $V\in\mathcal{U}(d)$, we obtain $0\in S'=\varphi^{-1}(S)\subset \mathcal{O}_1$.
    Moreover, $S'$ divides $\C^d$ into two connected open components, and the point $\varphi^{-1}(0)$ lies in the bounded component. 
    Now, for any $U,V\in\mathcal{U}(d)$ we have that $W_{\lambda,\varphi}\pare{W_{\lambda,\varphi}C_V}^{-1}C_U$ is a weighted composition automorphism of $\cH_1$ with composition symbol $U\circ \varphi^{-1}\circ V^{-1}\circ \varphi$.
    Thus, $\mathcal{U}(d)S'\subset \mathcal{O}_1$ and we can rotate $S'$ until we reach $\varphi^{-1}(0)$.  
    Hence, there are some $U,V\in\mathcal{U}(d)$ such that $\varphi^{-1}(0)\in US'\subset\mathcal{O}_1$ i.e. $\varphi^{-1}(0)=U\varphi^{-1}(V\varphi(0))$.
    Finally, we obtain that $W_{\lambda,\varphi}C_V W_{\lambda,\varphi}^{-1} C_U W_{\lambda,\varphi}$ is an invertible weighted composition with composition symbol $\varphi \circ U\circ \varphi^{-1}\circ V\circ \varphi$, which fixes the the origin.
\end{proof}

We now turn our attention to understanding the multiplier symbol.
For $t\in\R$, let $U_t$ denote the unitary map on $\C^d$ defined by $U_t(z) = e^{it}z$.
The following proposition can be generalized to spaces that are invariant under all unitaries of the form $U_t$, but the present formulation suffices for our purposes.
\begin{proposition}\label{prop : end}
    Let $\cH_i$ be a weighted Hardy space with the natural maximal domain $X_i$, $i=1,2$.
    Assume that $X_1\subseteq X_2\subseteq \C^d$ and let $\iota: X_1\hookrightarrow X_2$ be the inclusion map.
    If there exists $\lambda:X_1\to \C$ such that $\lambda(0)\neq 0$ and $W_{\lambda,\iota}:\cH_2\to\cH_1$ is bounded, then $C_\iota:\cH_2\to\cH_1$ is bounded and $\cH_2\subseteq \cH_1$ as sets of functions.
\end{proposition}
\begin{proof}
    We can assume that $\lambda(0)=1$. 
    Observe that $C_{U_t}W_{\lambda,\iota} C_{U_t}^{-1}=W_{\lambda\circ {U_t},\iota}$ with $\norm{W_{\lambda\circ {U_t},\iota}}\le \norm{W_{\lambda,\iota}}$ for any unitary $U_t$. 
    We also know that $\lambda $ is holomorphic on $\on{int}(X_1)$ since $W_{\lambda,\iota}(\textbf{1})=\lambda\in \cH_1$.
    For $z\in \on{int}(X_1)$ we see that:
    $$1= \lambda(0)=\int_0^1\lambda(e^{i2\pi t}z)dt.$$
    The map $t\mapsto \lambda(e^{i2\pi t}z)$ is continuous, so the integral is equal to a limit of Riemann sums:
    $$1=\int_0^1 \lambda(e^{i2\pi t}z)dt=\lim_n \frac{1}{n} \sum_{k=1}^n \lambda\circ U_{\nicefrac{2\pi k}{n}}(z).$$
    Set $\mu_n=\frac{1}{n}\sum_{k=1}^n \lambda\circ U_{\nicefrac{2\pi k}{n}}$, then $\norm{W_{\mu_n,\iota}}\le \norm{W_{\lambda,\iota}}$ for all natural $n$. 
    Fix $F\in \cH_2$. We shall show that exists $f\in \cH_1$ with $F\equiv f$ on $\on{int}(X_1)$. 
    Notice that $W_{\mu_n,\iota}(F)=\mu_nF\circ\iota\in \cH_1$ and $\norm{\mu_nF\circ\iota}_{\cH_1}\le \norm{W_{\mu_n,\iota}}\norm{F}_{\cH_2}\le \norm{W_{\lambda,\iota}}\norm{F}_{\cH_2}$, so, by Alaoglu's theorem, there is subsequence and $f\in\cH_1$ such that $\mu_{n_j}F\circ\iota\xrightarrow{w} f$.
    Since weak convergence implies pointwise convergence for any $z\in \on{int}(X_1)$ we have:
    $$f(z)=\lim_j \inp{\mu_{n_j}F\circ \iota,k^1_z}_{\cH_1}=\lim_j \mu_{n_j}(z)F\circ \iota(z)=\lim_j \mu_{n_j}(z)F(z)=F(z).$$
    Thus, $F=f$ on $\on{int}(X_1)$. By continuity of $F$ and $f$, we obtain that $f=F\circ\iota$, which implies that $\cH_2\subseteq \cH_1$ as sets of functions. 
    Finally, $C_\iota:\cH_2\to\cH_1$ is bounded by Theorem \ref{Thm : bound} with $\norm{C_\iota}\le \norm{W_{\lambda,\iota}}$.  
\end{proof}

We are now in a position to prove Theorem \ref{thm : big1}.
\begin{proof}[\textbf{Proof of Theorem \ref{thm : big1}}]
    Set $r_i\in (0, \infty]$ to be the radius of convergence of $G_i$, $i=1,2$.
    By Proposition \ref{cor : sr}, the natural maximal domains of $\cH_1$ and $\cH_2$ have the same dimension $d$, $r_2<\infty$, and we can assume that $W_{\lambda,\varphi}:\cH_2\to\cH_1$ is invertible with an inverse that is also a weighted composition operator.
    By Proposition \ref{lem : Orr} we may assume that $\varphi(0)=0$.
    
    Set $G(t)=G_2(\gamma t)=\sum_{n=0}^\infty b_n\gamma^{n}t^n$ with $\gamma=\nicefrac{r_2}{r_1}$.
    Denote by $\cH$ the weighted Hardy space in $d$ variables with $G$ as the generating function.
    Since $r_2$ is the radius of convergence of  $G_2$, we have that $r_1$ is the radius of convergence of $G$.
    
    Moreover, Lemma \ref{lem : kernel relation} yields that the map $C_{\sqrt{\gamma}z}:\cH\to\cH_2$ defined by $f(z)\mapsto f(\sqrt{\gamma} z)$ is unitary.
    Thus, $W_{\lambda,\psi}= W_{\lambda,\varphi}C_{\sqrt{\gamma}}:\cH\to\cH_1 $
    is invertible with $\psi(z)=\sqrt{\gamma}\varphi(z)$.
    Since $W_{\lambda,\psi}$ is invertible $\psi\in\on{Aut}(\sqrt{r_1}\bB_d)$ by Proposition \ref{prop : dim}.
    Notice that $\psi(0)=\sqrt{\gamma}\varphi(0)=0$, so $\psi\in\mathcal{U}(d)$ by Cartan's theorem.
    By the last part of Proposition \ref{cor : sr}, $\cH$ and $\cH_1$ have the same natural maximal domain.
    Hence,  $M_{\lambda}=W_{\lambda,\psi}C_{\psi^{-1}}:\cH\to\cH_1$ is invertible. 
    Thus, by Proposition \ref{prop : end}, the inclusion map $\cH\hookrightarrow \cH_1$ is bounded and there is some $c>0$ such that $\norm{z_1^n}_{\cH}^2\le c \norm{z_1^n}_{\cH_1}^2$ for all $n$. 
    So $a_n \le c\cdot b_n\gamma^n $ for all $n$.
    Since $M_{\lambda}^{-1}=M_{\nicefrac{1}{\lambda}}$ is also invertible, for the same reasons, there is some $C>0$ such that $\norm{z_1^n}_{\cH_1}^2\le C \norm{z_1^n}_{\cH}^2$ for all $n$.
    So $b_n\gamma^n \le C\cdot a_n$.
    Therefore,
    $$\frac{1}{c}\le \frac{b_n}{a_n}\gamma^n \le C$$
    for all $n$, and the proof is complete.
\end{proof}
For Theorem \ref{thm : big2}, we need the following lemma.
\begin{lemma}\label{prop : const}
    Let $\cH_1$ and $\cH_2$ be weighted Hardy spaces in one variable.
    Assume that the generating functions of both spaces are entire.
    If $W_{\lambda,\varphi}:\cH_2\to\cH_1$ is invertible with a composition symbol that fixes the origin, then $\lambda$ is a constant function. 
\end{lemma}
      
\begin{proof}
    Recall that $\on{Aut}(\C)=\set{az+b:a\neq 0}$, for example see \cite[Exercise 14 in Chapter 8]{stein2003complex}.
    By Proposition \ref{prop : dim}, $\varphi\in\Aut(\C)$ and the inverse of $W_{\lambda,\varphi}$ is $W_{\nicefrac{1}{\lambda\circ\varphi^{-1}},\varphi^{-1}}$.
    Thus, $\varphi(z)=az$ with $a\neq 0$ and for convenience we use $\delta(z)=(\lambda\circ\varphi^{-1}(z))^{-1}$.
    By boundedness of $W_{\lambda,\varphi}$ and its inverse there are positive $c_1,c_2$ such that: 
    \begin{gather*}
        \av{\lambda(z)}^2K_2(\varphi(z),\varphi(z))\le c_1 K_1(z,z) \\
        \av{\delta(w)}^2K_1(\varphi^{-1}(w),\varphi^{-1}(w))\le c_2K_2(w,w)
    \end{gather*}
    for all $z,w\in \C$. 
    Putting $w=\varphi(z)=az$ we obtain
    $$\frac{1}{c_1}\frac{K_2(w,w)}{K_1(\varphi^{-1}(w),\varphi^{-1}(w))}\le\av{\delta(w)}^2\le c_2\frac{K_2(w,w)}{K_1(\varphi^{-1}(w),\varphi^{-1}(w))}$$
    for all $w\in \C$.
    Since $\lambda\in\cH_1$ is entire, by Liouville's theorem, it is sufficient to show that $\lambda$ is bounded. 
    Assume, for the sake of reaching a contradiction, that $\av{\lambda(z_n)}\to \infty$ with $z_n\in \C$. Then, for $w_n=\varphi(z_n)=az_n$, we have $\av{\delta(w_n)}=\av{\frac{1}{\lambda\circ\varphi^{-1}(w_n)}}\to 0$.
    Let $\epsilon>0$ be arbitrary. 
    Since $\varphi$ is injective, there is some $n$ such that $\av{\delta(w_{n})}^2<\epsilon$ and $w_n\neq 0$.
    Notice that for any $t$ we have:
    $$\frac{K_2(U_tw_n,U_tw_n)}{K_1(U_t\varphi^{-1}(w_n),U_t\varphi^{-1}(w_n))}=\frac{K_2(w_n,w_n)}{K_1(\varphi^{-1}(w_n),\varphi^{-1}(w_n))}\le c_1\av{\delta(w_n)}^2< c_1\epsilon.$$
    Put $r=\av{w_n}$, then $\delta$ is holomorphic on $r\D$ and continuous on $\ol{r\D}$, so by the maximum modules principle there is unitary $U_t$ such that
    $$\sup_{w\in r\D} \av{\delta(w)}^2=\av{\delta(U_tw_n)}^2\ge\av{\delta(0)}^2.$$
    But now we have: 
    \begin{align*}
        \av{\delta(0)}^2 & \le\av{\delta(U_tw_n)}^2
        \le c_2\frac{K_2(U_tw_n,U_tw_n)}{K_1(\varphi^{-1}(U_tw_n),\varphi^{-1}(U_tw_n))} \\
        &=c_2\frac{K_2(U_tw_n,U_tw_n)}{K_1(\frac{e^{it}}{a}(w_n),\frac{e^{it}}{a}(w_n))}
        =c_2\frac{K_2(U_tw_n,U_tw_n)}{K_1(U_t\varphi^{-1}(w_n),U_t\varphi^{-1}(w_n))} \\
        & <c_1c_2\epsilon.
    \end{align*}
    Since $\epsilon$ is arbitrary $\delta(0)=0$. 
    However, by Lemma \ref{lem : inj} $\delta$ never vanishes since $W_{\delta,\varphi^{-1}}$ is invertible which leads to a contradiction.
    Hence, $\lambda$ is a constant function. 
\end{proof}

\begin{proof}[\textbf{Proof of Theorem \ref{thm : big2}}]
    By Theorem \ref{thm : big1}, it remains to consider the case when $G_1$ is entire.
    By Proposition \ref{cor : sr}, the natural maximal domain of $\cH_1$ and $\cH_2$ is $\C$, and we can assume that $W_{\lambda,\varphi}:\cH_2\to\cH_1$ is invertible with an inverse that is also a weighted composition operator.
    By Proposition \ref{lem : Orr} and Lemma \ref{prop : const}, we may assume that $\varphi\in\on{Aut}(\C)$ with $\varphi(0)=0$ and $\lambda=\textbf{1}$. 
    Hence, $W_{\lambda,\varphi}=C_{az}:f(z)\mapsto f(az)$ with $a\neq 0$. 
    
    Set $\gamma=\av{a}^2$; then $\gamma^n\norm{z^n}^2_{\cH_1}=\norm{a^nz^n}_{\cH_1}^2\le \norm{C_{az}}^2\norm{z^n}^2_{\cH_2}$, so $b_n\gamma^n \le Ca_n$ for all $n$.
    Also, $\gamma^{-n}\norm{z^n}^2_{\cH_2}=\norm{a^{-n}z^n}^2_{\cH_2}\le\norm{C_{az}^{-1}}^2\norm{z^n}^2_{\cH_1}$, so $a_n\le c b_n \gamma^n$ for all $n$. 
    Therefore,
    $$\frac{1}{c}\le \frac{b_n}{a_n}\gamma^n \le C$$
    for all $n$, and the proof is complete.
\end{proof}

This leads us to the immediate conclusion:
\begin{corollary}\label{cor : endbig1}
    Let $\cH_1$ be weighted Hardy space with the generating function $G_1(t)=\sum_{n=0}^\infty a_nt^n$ and $\cH_2$ be weighted Hardy space with the generating function $G_2(t)=\sum_{n=0}^\infty b_nt^n$. 
    Assume that $\cH_2$ is a one-variable function space or algebraically consistent on its natural maximal domain.
    If $\cH_1$ is weakly related to $\cH_2$, then there exist $C,c,\gamma\in (0,\infty)$ such that $c\le \frac{b_n}{a_n}\gamma^n\le C$ for all $n$.
\end{corollary}
\begin{proof}
    If $\cH_2$ is a weighted Hardy space in one variable, then the result follows from the first part of Proposition \ref{cor : sr} and Theorem \ref{thm : big2}.
    
    If $\cH_2$ is algebraically consistent on its natural maximal domain, then $G_2$ is not entire, and the result follows from Proposition \ref{prop : new} and Theorem \ref{thm : big1}.
\end{proof}

At present, the author does not have a solution for the case where $d\ge 2$ and the generating functions are entire, since $\on{Aut}(\C^d)$ is a wild set.
However, we also have a converse to Theorem \ref{thm : big1}
\begin{theorem}\label{thm : big3}
    Let $\cH_1$ and $\cH_2$ be weighted Hardy spaces in $d$ variables. 
    Assume that $G_1(t)=\sum_{n=0}^\infty a_nt^n$ is the generating function of $\cH_1$ and $G_2(t)=\sum_{n=0}^\infty b_nt^n$ is the generating function of $\cH_2$. 
    If there exist $C,c,\gamma\in (0,\infty)$ such that $c\le \frac{b_n}{a_n}\gamma^n\le C$ for all $n$, then $\cH_1$ and $\cH_2$ are isomorphic via an RKHS isomorphism.
\end{theorem}
\begin{proof}
    Let $r_i\in (0,\infty]$ be the radius of convergence of $G_i$, $i=1,2$.
    First, we observe that
    $$\frac{1}{r_1}=\limsup \sqrt[n]{ca_n}\le \limsup \gamma\sqrt[n]{b_n} =\frac{\gamma}{r_2}\le\limsup \sqrt[n]{Ca_n}=\frac{1}{r_1}$$
    so, $\gamma r_1=r_2$.
    Moreover, by symmetry $G_1(r_1)<\infty$ if and only if $G_2(r_2)<\infty$.
    
    If $f(z)=\sum_{n=0}^\infty f_n(z)\in \cH_2$ where $f_n(z)=\sum_{\av{\alpha}=n} f_\alpha z^\alpha$ is a homogeneous polynomial of degree $n$, then 
    $$\sum_{n=0}^\infty \sum_{\av{\alpha}=n} \av{f_\alpha}^2\frac{\alpha!}{n!}\frac{1}{b_n}=\sum_{n=0}^\infty \sum_{\av{\alpha}=n} \av{f_\alpha}^2\norm{z^\alpha}^2_{\cH_2}<\infty $$
    so, 
    $$\sum_{n=0}^\infty \gamma^n\sum_{\av{\alpha}=n} \av{f_\alpha}^2\frac{\alpha!}{n!}\frac{1}{a_n}<\infty$$
    which implies that $f(\sqrt{\gamma} z)=\sum_{n=0}^\infty \sqrt{\gamma}^n f_n(z)\in \cH_1$ with $\norm{f(\sqrt{\gamma}z)}_{\cH_1}^2\le C\norm{f(z)}_{\cH_2}^2$. 
    
    Hence, $C_{\sqrt{\gamma}z}:\cH_2\to\cH_1$ defined by $f(z)\mapsto f(\sqrt{\gamma} z)$ is bounded with $\norm{C_{\sqrt{\gamma}z}}^2\le C$ (Theorem \ref{Thm : bound}).
    Symmetrically, we obtain that $C_{\sqrt{\gamma}z}^{-1}=C_{\sqrt{\gamma^{-1}}z}$ is bounded and then $C_{\sqrt{\gamma}z}^\ast:\cH_1\to\cH_2$ is an RKHS isomorphism. 
\end{proof}
    
In the setting of the last theorems, we see that when $\gamma\ge 1$, the relation $c\le \frac{b_n}{a_n}\gamma^n\le C$ for all $n$ implies $X_1\subseteq X_2$ and the inclusion map $C_\iota:\cH_2\hookrightarrow\cH_1$ is bounded. 
Similarly, if $\gamma\leq 1$, then $X_2\subseteq X_1$ and the inclusion map $C_\iota:\cH_1\hookrightarrow\cH_2$ is bounded. 

\begin{example}\label{ex:1.1}
The following examples show that, in general, the existence of an inclusion map does not imply the existence of a weak relation, nor does the existence of a weak relation imply the existence of an inclusion map. 
    \begin{enumerate}   

        \item\label{ex:1.1.1} There exist isomorphic weighted Hardy spaces in one variable, such that one is strictly contained in the other and the spaces are strongly related.
        Let $G_0(t)=e^t$ be the generating function of the Segal--Bargmann space $\cS$, and let $\cH$ be the weighted Hardy space in one variable with the generating function $G_4(t)=G_0(4t)$. 
        So, $\cS\subseteq\cH$ and by Lemma \ref{lem : kernel relation}, $C_{2z}:\cS\to\cH$ defined by $f(z)\mapsto f(2z)$ is unitary.
        Moreover, the containment is strict. Indeed, if $\cH\subseteq \cS$, then by Theorem \ref{Thm : bound} the inclusion map $\cH\hookrightarrow \cS$ would be bounded. 
        Thus there would existis some $c>0$ such that $\norm{z^n}_{\cS}^2\le c \norm{z^n}_{\cH
        }^2$ for all $n$.
        Thus,  $\frac{4^n}{n!}=\norm{z^n}_{\cH}^{-2}\le c\norm{z^n}_{\cS}^{-2}=\frac{c}{n!}$ for all $n$, which of course can not be true.
        Using this method, one can find examples of spaces in any number of variables.  

         \item There exist weighted Hardy spaces in one variable, with finite radius of convergence, such that one is strictly contained in the other, but the two spaces are not weakly related. 
         For instance, the Hardy space $H^2(\D)$ is contained in the Bergman space $A^2(\D)$, and we saw that $H^2(\D)$ is not weakly related to $A^2(\D)$ (in fact, we can also conclude it from Corollary \ref{cor : endbig1}).
        Thus, in the finite-radius setting, strict containment does not imply even the weakest relation considered here. Similar examples can be constructed for any number of variables.
        
        \item There are also examples with entire generating functions.
        Let $\cS$ be the Segal--Bargmann space and $\cH$ be the weighted Hardy space in one variable with the generating function $G(t)=\sum_{n=0}^\infty \frac{1}{n!n^n}t^n$.
        Hence $\cH\subseteq\cS$.
        We claim that $\cH$ and $\cS$ are not weakly related. 
        Suppose, toward a contradiction, that such a relation exists. 
        In one variable, this yields an invertible composition operator $C_{az}:\cS\to\cH$ given by $f(z)\mapsto f(az)$ where $a\neq 0$. 
        Thus, for all $n$, we get
        $$(\av{a}^2)^nn!n^n=\norm{a^nz^n}_\cH^2\le \norm{C_{az}}^2\norm{z^n}^2_\cS=\norm{C_{az}}^2n!,$$
        which is possible only if $a=0$. 
        Hence, $\cH$ and $\cS$ are not weakly related. 

        Currently, it remains open whether, in dimensions $d\ge 2$, the assumptions $\cH_1\subseteq \cH_2$ and the entireness of the generating functions are sufficient to guarantee the existence of some relation.
    \end{enumerate}
\end{example}
We complete the section by dealing with the case where $W_{\lambda,\varphi}$ is a coisometry map. 
The Hilbert function spaces $\cH_1$ and $\cH_2$ are called \emph{isometricly isomorphic as RKHSs} if there exists an isometric RKHS isomorphism $T:\cH_1\to \cH_2$.
In this situation, the induced weighted composition operator $T^\ast$ is a coisometry.
Notice that when $W_{\lambda,\varphi}:\cH_2\to\cH_1$ is a unitary map between weighted Hardy spaces with the same number of variables, then $W_{\lambda,\varphi}^\ast$ is an isometric RKHS isomorphism.
We begin with simple and standard observations.
\begin{lemma}\label{lem : last}
    Let $\cH_1$ and $\cH_2$ be weighted Hardy spaces in $d$ variables.
    Assume that $X_1$ is the natural maximal domain of $\cH_1$.
    If $W_{\lambda,\varphi}:\cH_2\to\cH_1$ is a coisometry, then:
    \begin{enumerate}
        \item $K_2(\varphi(\cdot),\varphi(0))$ does not vanish and for any $x\in X_1$ we have
           $$K_2(\varphi(x),\varphi(0))=\frac{1}{\lambda(x)\ol{\lambda(0)}}.$$

        \item $W_{\lambda,\varphi}$ is unitary.

        \item  $\varphi(0)=0$ if and only if $\lambda$ is a constant function. In particular, in both cases $\av{\lambda}\equiv1$.
    \end{enumerate}
\end{lemma}
\begin{proof}
    For $x,y\in X_1$ we have 
    \begin{gather*}
        K_1(x,y)=\inp{k^1_y,k^1_x}_{\cH_1}=\inp{W_{\lambda,\varphi}^\ast k^1_y,W_{\lambda,\varphi}^\ast k^1_x}_{\cH_2}=\overline{\lambda(y)}\lambda(x)K_2(\varphi(x),\varphi(y)).
    \end{gather*}
    In particular, for $y=0$ we have $1=\overline{\lambda(0)}\lambda(x)K_2(\varphi(x),\varphi(0)).$
    Hence, $k^2_{\varphi(0)}\circ \varphi$ never vanishes and $(1)$ follows. 

    By Lemma \ref{lem : inj}, if $f\in\ker W_{\lambda,\varphi}$, then $f=0$ on the open set $\varphi(\on{int}(X_1))$, so by the identity theorem $f\equiv0$. Thus, $W_{\lambda,\varphi}$ is also injective and $(2)$ holds.
    
    For $(3)$, we have  
    $$W_{\lambda,\varphi}(k^2_{\varphi(0)})(x)=\lambda(x)k^2_{\varphi(0)}\circ \varphi(x)=\frac{1}{\overline{\lambda(0)}}$$
    for all $x\in X_1$.
    Hence, if $\varphi(0)=0$ then $k^2_{\varphi(0)}$ is a constant function and $\lambda$ is also constant and it is immediate that $\av{\lambda}\equiv1$.
    On the other hand, if $\lambda$ is a constant function,  $k^2_{\varphi(0)}$ must be a constant function by the injectivty of $W_{\lambda,\varphi}$.
    By 2-interpolating, we obtain that $k^2_{\varphi(0)}=k^2_0$ and $\varphi(0)=0.$
    Alternatively, if $\lambda$ is a constant function and $K_2(x,y)=G(\inp{x,y}_d)$, then $$G(\inp{\varphi(zx),\varphi(0)}_d)=\av{\lambda(0)}^{-2}$$ 
    for all $z\in\D$ and $x\in X_1$.
    Differentiating both sides with respect to $z$ and plugging in $0$ to obtain
    $$\inp{D\varphi(0)x,\varphi(0)}_d=0$$
    for all $x\in X_1$. Since $\varphi$ is injective $D\varphi(0)$ is invertible \cite[Theorem  1.6.6]{curry2025tasty}, so $\varphi(0)=0$.
\end{proof}

Although Theorems \ref{thm : big1} and \ref{thm : big2} require an assumption on the natural maximal domains, the situation is different for unitary equivalence: 
In that setting, no limitations on either the dimension or the radius of the natural maximal domain arise.
\begin{theorem}\label{thm : big 4}
    Let $\cH_1$ and $\cH_2$ be weighted Hardy spaces in $d$ variables.
    Assume that $G_1(t)=\sum_{n=0}^\infty a_nt^n$ is the generating function of $\cH_1$  and $G_2(t)=\sum_{n=0}^\infty b_nt^n$ is the generating function of $\cH_2$.
    The following are equivalent:
    \begin{enumerate}
        \item There exists a unitary $W_{\lambda,\varphi}:\cH_2\to\cH_1$.
        \item There exists a coisometry $W_{\lambda,\varphi}:\cH_2\to\cH_1$.
        \item There exists $\gamma\in(0,\infty)$ such that $a_n=b_n\gamma^n$ for all $n$.
    \end{enumerate} 
\end{theorem}
\begin{proof}
    The implication $(1)\Rightarrow(2)$ is immediate, and $(2)\Rightarrow(1)$ follows from Lemma \ref{lem : last}. Therefore, it remains to prove the equivalence between $(1)$ and $(3)$.

    $(3)\Rightarrow(1):$ If there exists $\gamma\in(0,\infty)$ such that $a_n=b_n\gamma^n$ for all $n$, then $G_1(t)=G_2(\gamma t)$ and $C_{\sqrt{\gamma}z} :\cH_2\to\cH_1$ is unitary by Lemma \ref{lem : kernel relation}.

    $(1)\Rightarrow (3):$ If $W_{\lambda,\varphi}:\cH_2\to\cH_1$ is a unitary map. We may assume that $\varphi(0)=0$ and $\lambda\equiv\textbf{1}$, so that $W_{\lambda,\varphi}=C_\varphi$.
    Indeed, for $U, V$ as in the proof of Proposition \ref{lem : Orr}, $W_{\lambda,\varphi}C_V W_{\lambda,\varphi}^{-1} C_U W_{\lambda,\varphi}$ is still unitary with composition symbol that fixes the origin, so by Lemma \ref{lem : last} $\lambda$ is a unimodular constant.
    Hence, $G_1(\inp{z,w}_d)=G_2(\inp{\varphi(z),\varphi(w)}_d)$.
    
    Since $\varphi$ is holomorphic and $\varphi(0)=0$, we may write $\varphi(z)=\sum_{\av{\alpha}\ge 1} \varphi_\alpha z^\alpha$ where $\varphi_\alpha\in \C^d$.
    Set $A=D\varphi(0)$, then $Az=\sum_{i=1}^d \varphi_{e_i}z_i$. Since $\varphi$ is injective, $A$ is invertible and then $\varphi_{e_1},\dots,\varphi_{e_d}$ form a basis of $\C^d$.
    Moreover, we have 
    $$\inp{\varphi(z),\varphi(w)}_d=\sum_{\av{\alpha},\av{\beta}\ge 1} \inp{\varphi_\alpha,\varphi_\beta}_dz^\alpha\ol{w}^\beta,$$
    and 
    \begin{equation}\label{eq:-2}
        \sum_{\alpha} a_{\av{\alpha}}\frac{\av{\alpha}!}{\alpha!}z^\alpha\ol{w}^\alpha=\sum_{n=0}^\infty b_n \pare{\sum_{\av{\alpha},\av{\beta}\ge 1} \inp{\varphi_\alpha,\varphi_\beta}_dz^\alpha\ol{w}^\beta}^n.
    \end{equation}
    By the uniqueness of the Taylor series, we can compare the coefficients of both sides at \eqref{eq:-2}.
    Fix a multi index $\alpha$ with $\av{\alpha}\ge 2$ and $i\in\set{1,\dots,d}$. 
    The coefficient of $z^\alpha\ol{w}^{e_i}$ on the left-hand side of \eqref{eq:-2} is zero. On the right-hand side, only the term corresponding to $n=1$ can contribute to this coefficient, since every antiholomorphic monomial in $\inp{\varphi(z),\varphi(w)}_d$ has a degree of at least one. 
    Hence
    $$b_1 \inp{\varphi_\alpha,\varphi_{e_i}}_dz^\alpha\ol{w}^{e_i}=0.$$
    Since $b_1\neq 0$, it follows that $\inp{\varphi_\alpha,\varphi_{e_i}}=0$ for every $i\in\set{1,...,d}$.
    As $\varphi_{e_1},\dots,\varphi_{e_d}$ span $\C^d$, we obtain that $\varphi_\alpha=0$ for every $\av{\alpha}\ge 2$.
    Thus, $\varphi(z)=Az$ and $a_1\inp{z,w}_d=b_1\inp{Az,Aw}_d$ (by comparing the coefficient of $z^{e_i}\ol{w}^{e_i}$ or using the chain rule).
    Therefore, $\inp{Az,Aw}_d=\gamma\inp{z,w}_d$ where $\gamma=\frac{a_1}{b_1}$ and $G_1(\inp{z,w}_d)=G_2(\gamma\inp{z,w}_d)$. Finally, another comparison of the coefficients implies that $a_n=b_n\gamma^n$ for all $n$.
\end{proof}
The assumption that the spaces have the same number of variables is necessary. Some examples of coisometries which are not unitary are closely related to complete Pick spaces and \cite[Theorem 8.2]{Jim}. Those examples can be deduced from \cite[Chapter 3]{mironov2024hilbert}, which also inspired the proof of the last theorem.
For further classification results on the symbols of coisometric weighted composition operators $W_{\lambda,\varphi}:\cH\to\cH$, we refer the reader to \cite{hartz2025weighted,le2012self,martin2019co}. 
\begin{remark}
    In Theorem \ref{thm : gilad}, Ofek and Sofer require a composition symbol which is an automorphism of the disk. 
    Thus, after using the ``disc trick'', if necessary, and by Cartan's uniqueness theorem, we obtain an invertible weighted composition operator whose composition symbol is the identity map; that is, a multiplication operator.
    Inspecting the preceding arguments, we see that their theorem remains valid for general RKHS isomorphisms, namely, to operators of the form $T^\ast=W_{\lambda,\varphi}$ with $\lambda$ not necessarily constant. 
\end{remark}

\section{Application to the Banach--Mazur distance}
In \cite{OFEK}, the authors introduced an invariant of the Banach--Mazur distance for Hilbert function spaces and multiplier algebras. 
In this section, we apply our results to obtain a legitimate distance on strongly related weighted Hardy spaces and an estimate on this distance in terms of the relations between the weights of the generating function.  

\begin{definition}
    Let $\cH_1$ and $\cH_2$ be weighted Hardy spaces which are isomorphic via an RKHS isomorphism (equivalently, strongly related). Define
    $$\delta_{\on{RK}}\pare{\cH_1,\cH_2}=\inf\set{\norm{T}\norm{T^{-1}}: T:\cH_1\to\cH_2\text{ is an RKHS isomorphism}}.$$
    The \emph{reproducing kernel Banach-Mazur distance} between $\cH_1$ and $\cH_2$ is defined to be
    $$\rho_{\on{RK}}\pare{\cH_1,\cH_2}=\log\pare{\delta_{\on{RK}}\pare{\cH_1,\cH_2}}.$$
\end{definition}
Theorems \ref{thm : big1} and \ref{thm : big3} induce an equivalence relation on the class of weighted Hardy spaces in $d$ variables. We can use Theorem \ref{thm : big3} and its proof to obtain an upper bound:
\begin{lemma}\label{lem:ed}
    Let $\cH_1$ and $\cH_2$ be weighted Hardy spaces in $d$ variables. 
    Assume that $G_1(t)=\sum_{n=0}^\infty a_nt^n$ is the generating function of $\cH_1$ and $G_2(t)=\sum_{n=0}^\infty b_nt^n$ is the generating function of $\cH_2$. 
    If there exist $C,c,\gamma\in (0,\infty)$ such that $c\le \frac{b_n}{a_n}\gamma^n\le C$ for all $n$, then 
        $$\delta_{\on{RK}}(\cH_1,\cH_2)^2\le \frac{C}{c}.$$
    In particular, $\delta_{\on{RK}}(\cH_1,\cH_2)\le \max\set{C,\frac{1}{c}}$.
\end{lemma}
\begin{proof}
    By the proof of Theorem \ref{thm : big3} we see that the map $C^\ast_{\sqrt{\gamma}z} :\cH_1\to\cH_2$ is an RKHS isomorphism with inverse $C^\ast_{\sqrt{\gamma^{-1}}z} :\cH_2\to\cH_1$. 
    In addition, we saw in the proof that $\norm{C^\ast_{\sqrt{\gamma}z}}^2\le C$.  Similarly, we can deduce that $\norm{C^\ast_{\sqrt{\gamma^{-1}}z}}^2\le \frac{1}{c}$.  
    Thus, by definition of $\delta_{\on{RK}}$ we obtain
    $$\delta_{\on{RK}}(\cH_1,\cH_2)^2\le\norm{C_{\sqrt{\gamma}z}}^2\norm{C_{\sqrt{\gamma}z}^{-1}}^2\le \frac{C}{c}$$
\end{proof}
However, Theorem \ref{thm : big 4} reveals that $\rho_{\on{RK}}$ is not a distance function on those equivalence classes (one can find infinitely many different spaces which are isometrically strongly related).
Therefore, we need to restrict ourselves to smaller sets.

For a weighted Hardy space $\cH$ in $d$ variables and radius $r\in(0,\infty)$, we denote by $[\cH ]_r$ the equivalence class of all weighted Hardy spaces in $d$ variables and radius $r$ which are isomorphic via an RKHS isomorphism to $\cH$. 

Notice that by the proof of Theorem \ref{thm : big1}, for every $\cK\in[\cH]_r$, the identity map $C_{\on{id}}:\cH\to\cK$ is an invertible weighted composition operator.
Moreover, using the orthogonality of the monomials, it is easy to verify that
$$\norm{C_{\on{id}}}=\sup_n\frac{\norm{z_1^n}_\cK}{\norm{z_1^n}_\cH}.$$
The normalization $\norm{\textbf{1}}_{\cH}=\norm{\textbf{1}}_{\cH}=1$ implies that $\norm{C_{\on{id}}}\ge 1$. 
For a lower bound, we have: 
\begin{lemma}\label{lem:dis}
    Let $\cH_1$ and $\cH_2$ be elements in $[\cH]_r$. Denote by $I:\cH_2\to \cH_1$ the identity map, then 
    $$\norm{I}\norm{I^{-1}}\le\delta_{\on{RK}}\pare{\cH_1,\cH_2}^3.$$
    In particular, $\min\set{\norm{I}^2,\norm{I^{-1}}^2}\le \delta_{\on{RK}}\pare{\cH_1,\cH_2}^3$ and $\max\set{\norm{I},\norm{I^{-1}}}\le \delta_{\on{RK}}\pare{\cH_1,\cH_2}^3$.
\end{lemma}
\begin{proof}
    Supposed that $W_{\lambda,\varphi}:\cH_2\to\cH_1$ is invertible.
    As both spaces have radius $r$, we see that $\varphi\in\on{Aut}(r\bB_d)$.
    By the ``disc trick'' (Proposition \ref{lem : Orr}) and Cartan's theorem, there exist $R,U,V\in\mathcal U(d)$ such that 
    $$M_\mu = C_RW_{\lambda,\varphi}C_UW_{\lambda,\varphi}^{-1}C_VW_{\lambda,\varphi}:\cH_2\to\cH_1$$
    is an invertible multiplication operator. 
    Since the operators $C_R,C_U,C_V$ are unitary we obtain
    \begin{equation}\label{eq:1}
      \norm{M_\mu}\le \norm{W_{\lambda,\varphi}}^2\norm{W_{\lambda,\varphi}^{-1}}.  
    \end{equation}
    Moreover, $M_\mu^{-1}=W_{\lambda,\varphi}^{-1} C_V^{-1}W_{\lambda,\varphi}C_U^{-1}W_{\lambda,\varphi}^{-1}C_R^{-1}$, so
    \begin{equation}
        \norm{M_\mu^{-1}}\le \norm{W_{\lambda,\varphi}}\norm{W_{\lambda,\varphi}^{-1}}^2.
    \end{equation}
    Since $M_\mu$ is invertible, $\mu(0)\neq 0$ and by the notation at the proof of Proposition \ref{prop : end} we observe that $I=C_\iota$ with 
    \begin{equation}
        \av{\mu(0)}\norm{I}\le \norm{M_\mu}
    \end{equation}
    Applying the same result to $M_\mu^{-1}=M_{\nicefrac{1}{\mu}}$ gives
    \begin{equation}\label{eq:4}
        \frac{1}{\av{\mu(0)}}\norm{ I^{-1}}\le \norm{M_\mu^{-1}} 
    \end{equation}
    Combining \eqref{eq:1}-\eqref{eq:4}, we obtain 
    \begin{equation}\label{eq:5}
      \norm{I}\norm{I^{-1}}\le \norm{M_\mu}\norm{M_\mu^{-1}}\le \pare{\norm{W_{\lambda,\varphi}}\norm{W_{\lambda,\varphi}^{-1}}}^3.  
    \end{equation}
    Since every RKHS isomorphism $T:\cH_1\to\cH_2$ induces a weighted composition operator $T^\ast:\cH_2\to\cH_1$ with the same operator norm and \eqref{eq:5} holds for every invertible weighted composition operator, taking the infimum over all such operators yields 
    $$\min\set{\norm{I}^2,\norm{I^{-1}}^2}\le\norm{I}\norm{I^{-1}}\le\delta_{\on{RK}}\pare{\cH_1,\cH_2}^3.$$
    The last estimate is obtained by the normalization of $\textbf{1}$.
\end{proof}
Using Theorem \ref{thm : big 4} and the results in \cite{hartz2025weighted}, one can classify all the weighted Hardy spaces in $d$ variables and finite radius with the property that for every automorphism $\psi$ of its natural maximal domain, there exists a self-unitary weighted composition operator with composition symbol $\psi$. 
For these types of spaces, we can find a sharp estimate for $\delta_{\on{RK}}$.
\begin{proposition}
    Assume that $\cH$ is a weighted Hardy space in $d$ variables and radius $r\in(0,\infty)$.
    Suppose that $\cH$ has the property that for every automorphism $\psi$ of its natural maximal domain, there exists a unitary weighted composition operator $W_{\delta,\psi}:\cH\to\cH$.
    Let $\cK$ be an element in $[\cH]_r$ and denote by $I:\cH\to\cK$ the identity map, then
    $$\norm{I}\norm{I^{-1}}=\delta_{\on{RK}}(\cH,\cK).$$
\end{proposition}
\begin{proof}
    Suppose that $W_{\lambda,\varphi}:\cH\to\cK$ is invertible. Since $\varphi^{-1}\in \on{Aut}(r\bB_d)$, there exsits a unitary map $W_{\delta,\varphi^{-1}}:\cH\to\cH$. 
    Hence, $M_\mu=W_{\lambda,\varphi}W_{\delta,\varphi^{-1}}$ is invertible with $\norm{M_\mu}=\norm{W_{\lambda,\varphi}}$.
    We obtain as above
    $$\norm{I}\norm{I^{-1}}\le \norm{W_{\lambda,\varphi}}\norm{W_{\lambda,\varphi}^{-1}}.$$
    Thus, the result easily follows from the definition of $\delta_{\on{RK}}$.
\end{proof}
\begin{quest}
    Can we always find a minimizing sequence converging to $\delta_{\on{RK}}$ consisting of multiplication operators?
\end{quest}

\begin{corollary}\label{cor:end}
    $\rho_{\on{RK}}$ defines a metric on $[\cH]_r$.
\end{corollary}
\begin{proof}
    The only non trivial part is to show that $\rho_{\on{RK}}(\cH_1,\cH_2)=0$ implies that $\cH_1$ and $\cH_2$ are the same space. 
    Assume that $\rho_{\on{RK}}(\cH_1,\cH_2)=0$, then $\delta_{\on{RK}}(\cH_1,\cH_2)=1$.
    By Lemma \ref{lem:dis}, we obtain that the identity map between $\cH_1$ and $\cH_2$ is a unitary map. 
\end{proof}
\begin{remark}
    Notice that in the first example in Example \ref{ex:1.1}, we find two different spaces that are isometrically strongly related.
    Hence, $\rho_{\on{RK}}$ does not define a metric on $[\cH]_\infty$ with the analog definition.
\end{remark}
We finish the section by combining the estimates we find for the Banach-Mazur distance.
\begin{theorem}\label{thm:est}
    Let $\cH_1$ and $\cH_2$ be elements in $[\cH]_r$.
    Assume that $G_1(t)=\sum_{n=0}^\infty a_nt^n$ is the generating function of $\cH_1$ and $G_2(t)=\sum_{n=0}^\infty b_nt^n$ is the generating function of $\cH_2$.
    Set $M=\max\set{\sup_n \frac{b_n}{a_n},\sup_n\frac{a_n}{b_n}}$ and $m=\min\set{\sup_n \frac{b_n}{a_n},\sup_n\frac{a_n}{b_n}}$, then
    $$\max\set{\frac{1}{6}\log M,\frac{1}{3}\log m}\le \rho_{\on{RK}}(\cH_1,\cH_2)\le\frac{1}{2}\log Mm.$$
\end{theorem}
\begin{proof}
    Let $I:\cH_2\to\cH_1$ be the identity map.
    Observe that, $M=\max\set{\norm{I}^2,\norm{I^{-1}}^2}$ and $m=\min\set{\norm{I}^2,\norm{I^{-1}}^2}$. 
    Thus, $\rho_{\on{RK}}(\cH_1,\cH_2)\le\frac{1}{2}\log Mm$ by Lemma \ref{lem:ed}.
    Lemma \ref{lem:dis} shows that $\max\set{M^2,m}\le \delta_{\on{RK}}(\cH_1,\cH_2)^3$ which implies that $\max\set{\frac{1}{6}\log M,\frac{1}{3}\log m}\le \rho_{\on{RK}}(\cH_1,\cH_2)$ and finishes the proof. 
\end{proof}

\section{Open problems}
We conclude by highlighting several open questions.
These questions reflect the directions which, in the author's view, are the most natural next
steps in understanding the relations studied in this paper. 
\begin{problem}
The following problems remain open:
    \begin{itemize}
        \item Does a weak relation between weighted Hardy spaces imply a strong relation?

        \item Classify the faithful subsets of the natural maximal domain for a general weighted Hardy space.
        
        \item Classify weighted Hardy spaces of entire functions with more than one variable.

        \item Can we find a better estimates for $\rho_{\on{RK}}$?
    \end{itemize}
\end{problem}
Let us add some motivation for these questions.
The author believes that weighted Hardy spaces in different numbers of variables cannot be weakly related. If this is the case, then any weak relation between weighted Hardy spaces would already force a strong relation. 
From this point of view, classifying faithful subsets seems a promising approach to the problem.

\subsection*{Acknowledgements.} 
I want to express my sincere gratitude to Orr Shalit for his guidance, patience, and tremendous support throughout this work. 
This paper grew out of my master's thesis, which I wrote under his supervision.

\bibliographystyle{plain}
\bibliography{references.bib}

\end{document}